\documentclass[a4paper, 12pt]{article}

\usepackage[utf8x]{inputenc}
\usepackage{amsmath,amssymb,amsfonts}
\usepackage{amsthm}
\usepackage[abbrev]{amsrefs}
\usepackage[top=20truemm,bottom=20truemm,left=20truemm,right=20truemm]{geometry}
\usepackage{empheq}
\usepackage{color}
\usepackage[title]{appendix}

\numberwithin{equation}{section}
\theoremstyle{plain}
\newtheorem{theorem}{Theorem}[section]

\newtheorem{corollary}[theorem]{Corollary}
\newtheorem{proposition}[theorem]{Proposition}
\theoremstyle{definition}

\newtheorem{remark}[theorem]{Remark}

\newcommand{\Fourier}{ \mathcal{F}}
\newcommand{\FourierInverse}{ \mathcal{F}^{ - 1 } }

\newcommand{\Integer}{ \mathbb{Z} }

\newcommand{\Real}{ \mathbb{R} }

\newcommand{\Torus}{ \mathbb{T} }

\title{JUSTIFICATION OF THE HYDROSTATIC APPROXIMATION OF THE PRIMITIVE EQUATIONS IN ANISOTROPIC SPACE $L^\infty_H L^q_{x_3}(\Torus^3)$}

\author{Ken Furukawa \thanks{Institute of physical and chemical research (RIKEN), ken.furukawa@riken.jp} and Takahito Kashiwabara \thanks{The university of Tokyo, tkashiwa@ms.u-tokyo.ac.jp} \thanks{The second author was partly supported by JSPS Grant-in-Aid for Early-Career Scientists (No. 20K14357) and by Grant for The University of Tokyo Excellent Young Researchers.}
}

\begin{document}

\maketitle

\abstract{
    The primitive equations are fundamental models in geophysical fluid dynamics and derived from the scaled Navier-Stokes equations.
    In the primitive equations, the evolution equation to the vertical velocity is replaced by the so-called hydrostatic approximation.
    In this paper, we give a justification of the hydrostatic approximation by the scaled Navier-Stoke equations in anisotropic spaces $L^\infty_H L^q_{x_3} (\Torus^3)$ for $q \geq 1$.
}


\section{Introduction} \label{Sec3-1}
    We consider the primitive equations
    \begin{align} \label{eq_primitive}
            \begin{array}{rclcl}
                \partial_t v - \Delta v + v \cdot \nabla_H v + w \partial_3 v + \nabla_H \pi
                & =
                & 0
                & \mathrm{in}
                & \Torus^3 \times (0, \infty), \\
                \partial_3 \pi
                & =
                & 0
                & \mathrm{in}
                & \Torus^3 \times (0, \infty), \\
                \mathrm{div}_H \, v + \partial_3 w
                & =
                & 0
                & \mathrm{in}
                & \Torus^3 \times (0, \infty),\\
                v (0)
                & =
                & v_0
                & \mathrm{in}
                & \Torus^3,\\
            \end{array}
    \end{align}
    where $u = (v, w) \in \Real^2 \times \Real$ is a vector field, $\pi$ is a scalar function, $\nabla_H=(\partial_1, \partial_2)^T$ is the horizontal gradient, and $\Torus = \Real / 2 \pi \Integer$ is the flat torus.
    We impose the periodic boundary conditions.
    The second equation in (\ref{eq_primitive}) is called the hydrostatic approximation.
    The vertical component $w$ is give by the formula
    \begin{align} \label{eq_formula_for_w}
        w (t, x^\prime, x_3)
        = - \int_{- \pi }^{x_3}
            \mathrm{div}_H \, v (t, x^\prime , \zeta)
        d z, \quad
        x = (x^\prime, x_3) \in \Torus^2 \times \Torus, \, t>0,
    \end{align}
    where $\mathrm{div}_H = \nabla_H \cdot$ is the horizontal divergence.
    This formula is from the divergence-free condition.
    We also impose $w (\cdot, \cdot, \pm \pi)=0$.
    We always assume this assumption to the horizontal component of three-dimensional divergence-free vector fields in this paper.
    In the case of the Neumann boundary conditions in the domain $\Torus^2 \times (-\pi, 0)$ are reduced to the periodic case in $\Torus^3$ by the even and odd extension for $v$ and $w$, respectively.
    We invoke that $w$ satisfies the nonlinear parabolic equation
    \begin{align} \label{eq_for_w}
        \begin{split}
            \partial_t w - \Delta w
            & = \int_{- \pi}^{x_3}
                \mathrm{div}_H \, \left(
                    \tilde{v} \cdot \nabla_H \tilde{v}
                \right)
            d z
            + \mathrm{div}_H \left(
                \int_{- \pi}^{x_3}
                    \tilde{v}
                d z \cdot \nabla_H \overline{v}
            \right) \\
            & + \mathrm{div}_H \, \left(
                \overline{v} \cdot \nabla_H \int_{- \pi}^{x_3}
                    \tilde{v}
                d z
            \right)
            + \mathrm{div}_H (w \tilde{v}) \\
            & + \mathrm{div}_H \left(
                \int_{- \pi}^{x_3}
                    (\mathrm{div}_H \tilde{v}) \tilde{v}
                d z
            \right) \\
            &- \frac{1}{2} (x_3 - \pi) \mathrm{div}_H \int_{- \pi}^{x_3}
                \tilde{v} \cdot \nabla_H \tilde{v}
                + (\mathrm{div}_H \, \tilde{v}) \tilde{v}
            d z \\
            & =: F (v, w),
        \end{split}
    \end{align}
    for $x_3 \in (- \pi, \pi)$, where $\overline{v} = \frac{1}{2 \pi} \int_{- \pi}^\pi v \, dz $ and $\tilde{v} = v - \overline{v}$.
    Note that (\ref{eq_for_w}) differs from the equation for $w$ derived in Proposition 4.6 of \cite{FurukawaGigaHieberHusseinKashiwabaraWrona2018} in that no vertical derivatives appear in the right-hand-side terms.
    The derivation is described in Appendix \ref{seq_derivation_PE}.

    Existence of the global weak solution to the primitive equations on spherical shells was proved by Lions, Temam and Wang \cite{LionsTemamWangShou1992}.
    Local well-posedness was proved by Guill\'{e}n-Gonz\'{a}lez, Masmoudi and Rodr\'{i}guez-Bellido \cite{GuillenMasmoudiRodriguez2001} for $H^1$ initial data, where $H^s$ is the Sobolev space with $s \in \Real$.
    Cao and Titi \cite{CaoTiti2007} proved a $H^1$ energy bound to establish the global well-posedness.
    Hieber and Kashiwabara \cite{HieberKashiwabara2016} extended this result and proved the global well-posedness in Lebesgue spaces $L^p$-settings for $p\geq 3$.
    Recently, Giga, Gries, Hieber, Hussein, and Kashiwabara \cite{GigaGriesHieberHusseinKashiwabara2017_analiticity} showed the global well-posedness in $L^p$-$L^q$ settings under the periodic, Neumann, Dirichlet, Dirichlet-Neumann mixed boundary conditions.
    Giga, Gries, Hieber, Hussein, and Kashiwabara \cite{GigaGriesHieberHusseinKashiwabara2017_L_infty_L1} showed the global well-posedness in $L^\infty_H L^1_{x_3} (\Torus^3)$, where $L^\infty_H L^p_{x_3} (\Torus^3)$ for $q \geq 1$ denotes an anisotropic space equipped with the norm
    \begin{align} \label{eq_def_norm_Linfty_Lq}
        \Vert
            f
        \Vert_{L^\infty_H L^q_{x_3} (\Torus^3)}
        := \sup_{x^\prime \in \Torus^2} \left(
            \int_{\Torus}
                \left|
                    f (x^\prime, x_3)
                \right|^q
            d x_3
        \right)^{1/q}.
    \end{align}
    Giga, Gries, Hieber, Hussein, and Kashiwabara \cite{GigaGriesHieberHusseinKashiwabara2020_L_infty_Lp} also proved the global well-posedness in $L^\infty_H L^q_{x_3} (\Torus^2 \times (-h,0))$ for $h>0$ and $q \geq 3$ under the Dirichlet-Neumann mixed boundary conditions.
    An advantage of $L^\infty_H L^q_{x_3}$-approach is that one need not assume smoothness for initial data.

    The aim of this paper is to give a mathematically rigorous justification of the hydrostatic approximation for the primitive equations under less smoothness assumptions than the previous works.
    We first introduce a brief derivation of the primitive equations.
    The primitive equation is derived by the Navier-Stokes equations with anisotropic viscosity, which are horizontally $O(1)$ and vertically $O(\varepsilon^2)$.
    Applying a scaling to equations, we obtain the scaled Navier-Stokes equations
    \begin{align} \label{eq_scaled_Navier-Stokes}
        \left.
            \begin{array}{rclcl}
                \partial_t v_\varepsilon - \Delta v_\varepsilon + u_\varepsilon \cdot \nabla v_\varepsilon + \nabla_H \pi_\varepsilon
                & =
                & 0
                & \mathrm{in}
                & \Torus^3 \times (0, \infty), \\
                \varepsilon \left(
                    \partial_t w_\varepsilon - \Delta w_\varepsilon +   u_\varepsilon \cdot \nabla w_\varepsilon
                \right)
                + \partial_3 \pi_\varepsilon / \varepsilon
                &=
                & 0
                & \mathrm{in}
                & \Torus^3 \times (0, \infty), \\
                \mathrm{div} \, u
                & =
                & 0
                & \mathrm{in}
                & \Torus^3 \times (0, \infty), \\
                u_\varepsilon (0)
                & =
                & u_0
                & \mathrm{in}
                & \Torus^3, \\
            \end{array}
        \right.
    \end{align}
    see \cite{FurukawaGigaHieberHusseinKashiwabaraWrona2018}, \cite{FurukawaGigaKashiwabara2021}, and \cite{LiTiti2017} for the details.
    If we multiply $\varepsilon$ to the seconde equation of (\ref{eq_scaled_Navier-Stokes}) and take formal limit $\varepsilon \rightarrow 0$, then we obtain (\ref{eq_primitive}).

    To justify this formal derivation we have to show the difference between the solutions to (\ref{eq_primitive}) and (\ref{eq_scaled_Navier-Stokes}) converges to zero in some topologies.
    We put
    \begin{align*}
        \begin{split}
            & U_\varepsilon
            = (V_\varepsilon, W_\varepsilon), \\
            & V_\varepsilon
            = v_\varepsilon - v, \quad
            W_\varepsilon
            = w_\varepsilon - w, \quad
            \Pi_\varepsilon
            = \pi_\varepsilon - \pi.
        \end{split}
    \end{align*}
    Then we see that $(U_\varepsilon, \Pi_\varepsilon)$ satisfies
    \begin{align} \label{eq_difference}
        \left.
            \begin{array}{rclcl}
                \partial_t V_\varepsilon - \Delta V_\varepsilon + \nabla_H \pi_\varepsilon
                & =
                & F_H (U_\varepsilon, u)
                & \mathrm{in}
                & \Torus^3 \times (0, \infty),\\
                \varepsilon \left(
                    \partial_t W_\varepsilon - \Delta W_\varepsilon
                \right)
                + \partial_3 \Pi_\varepsilon / \varepsilon
                &=
                & \varepsilon F_3 (U_\varepsilon, u) + \varepsilon \tilde{F} (v, w)
                & \mathrm{in}
                & \Torus^3 \times (0, \infty), \\
                \mathrm{div} \, U_\varepsilon
                & =
                & 0
                & \mathrm{in}
                & \Torus^3 \times (0, \infty), \\
                U_\varepsilon (0)
                & =
                & 0
                & \mathrm{in}
                & \Torus^3,
            \end{array}
        \right.
    \end{align}
    where
    \begin{align} \label{eq_def_of_F}
        \begin{split}
            F_H (U_\varepsilon, u)
            & = - \left(
                U_\varepsilon \cdot \nabla V_\varepsilon + u \cdot \nabla   V_\varepsilon + U_\varepsilon \cdot \nabla v
            \right), \\
            F_3 (U_\varepsilon, u)
            & = - \left(
                U_\varepsilon \cdot \nabla W_\varepsilon + u \cdot \nabla   W_\varepsilon + U_\varepsilon \cdot \nabla w
            \right), \\
            \tilde{F} (v,w)
            & = - \left(
                F (v,w) + u \cdot \nabla w
            \right).
        \end{split}
    \end{align}
    Note that
    \begin{align*}
        \mathrm{div} \, U_\varepsilon
        = \mathrm{div}_H \, V_\varepsilon
        + \frac{\partial_3}{\varepsilon} (\varepsilon W_\varepsilon)
        = \mathrm{div}_\varepsilon (V_\varepsilon, \varepsilon W_\varepsilon)^T,
    \end{align*}
    where $\mathrm{div}_\varepsilon \, f = \mathrm{div}_H f^\prime +\partial_3 f_3 / \varepsilon$ for a vector field $f = (f^\prime, f_3)^T$.

    The justification of the hydrostatic approximation is reduced to showing that $U_\varepsilon$ converges to zero.
    Az\'{e}rad and Guill\'{e}n \cite{AzeradGuillen2001} proved the weak convergence in the energy space.
    Li and Titi \cite{LiTiti2017} showed the strong convergence in the energy space.
    They also proved global well-posedness to (\ref{eq_scaled_Navier-Stokes}) for small $\varepsilon$ compared to the initial data.
    The authors together with Giga, Hieber, Hussein, and Wrona \cite{FurukawaGigaHieberHusseinKashiwabaraWrona2018} extended Li and Titi's result to $L^p$-$L^q$ settings under the Neumann boundary conditions.
    The authors together with Giga \cite{FurukawaGigaKashiwabara2021} showed the strong convergence in $L^p$-$L^q$ settings under the Dirichlet boundary conditions.

    We consider the solution to (\ref{eq_difference}) in the sense of a mild solution, namely
    \begin{align} \label{eq_integral_form_diff_eq}
        \left(
            \begin{array}{c}
                V_\varepsilon (t) \\
                \varepsilon W_\varepsilon (t)
            \end{array}
        \right)
        = \int_0^t
        e^{(t-s)\Delta} \mathbb{P}_\varepsilon \left(
                \begin{array}{c}
                    F_H (U_\varepsilon (s), u (s))\\
                    \varepsilon F_3 (U_\varepsilon (s), u (s))
                    + \varepsilon \tilde{F} (v (s), w (s))
                \end{array}
            \right)
        ds,
    \end{align}
    where $\mathbb{P}_\varepsilon$ is the anisotropic Helmholtz projection which maps from $L^\infty_H L^q_{x_3} (\Torus^3)$-vector fields to $\mathrm{div}_\varepsilon$-free $L^\infty_H L^q_{x_3} (\Torus^3)$-vector fields.

    The first main result of this paper is the global well-posedness to (\ref{eq_difference}) in $L^\infty_H L^q_{x_3}(\Torus^3)$ setting for small $\varepsilon$ compared to the initial data of the primitive equations.
    We write $C_H L^q_{x_3} (\Torus^3) = C (\Torus ; L^q (\Torus^2))$ equipped with the norm (\ref{eq_def_norm_Linfty_Lq}) and $C_t C_H L^q_{x_3} (\Torus^3 \times I) = C(I ; C_H L^q_{x_3}(\Torus^3))$ equipped with the norm
    \begin{align} \label{eq_def_norm_C_t_Linfty_Lq}
        \Vert
            f
        \Vert_{C_t L^\infty_H L^q_{x_3} (\Torus^3 \times I)}
        := \sup_{t \in I} \sup_{x^\prime \in \Torus^2} \left(
            \int_{\Torus}
                \left|
                    f (x^\prime, x_3)
                \right|^q
            d x_3
        \right)^{1/q}.
    \end{align}
    for $q \geq 1$ and an interval $I$.

    \begin{theorem} \label{thm_main}
        Let $T>0$, $q \geq 1$, $u_0 = (v_0, w_0 )\in C_H L^q_{x_3} (\Torus^3)$ satisfy $\mathrm{div} \, u_0 = 0$ and $\nabla_H v_0 \in L^\infty_H L^q_{x_3} (\Torus^3)$, and $\varepsilon > 0$.
        Let $u \in C_t C_H L^q_{x_3} (\Torus^3 \times [0, T))$ be a solution to (\ref{eq_primitive}) with initial data $u_0$.
        Then there exists $\varepsilon_0 > 0$ such that, if $\varepsilon < \varepsilon_0$ the equation of the differences (\ref{eq_difference}) admits a unique solution $(V_\varepsilon, W_\varepsilon) \in C_t C_H L^q_{x_3} (\Torus^3 \times [0, T))$
        \begin{align} \label{eq_Fujita_Kato_estimate}
            & \sup_{0<t<T} \Vert
                V_\varepsilon
            \Vert_{ L^\infty_H L^q_{x_3} (\Torus^3)}
            + \sup_{0<t<T} t^{q/2} \Vert
                \nabla V_\varepsilon
            \Vert_{ L^\infty_H L^q_{x_3} (\Torus^3)} \notag \\
            & + \sup_{0<t<T} \Vert
                \varepsilon W_\varepsilon
            \Vert_{L^\infty_H L^q_{x_3} (\Torus^3)}
            + \sup_{0<t<T} t^{q/2} \Vert
                \varepsilon \nabla W_\varepsilon
            \Vert_{L^\infty_H L^q_{x_3} (\Torus^3)} \notag \\
            & \leq C \varepsilon,
            \end{align}
        where $C$ is independent of $\varepsilon$.
    \end{theorem}

    \begin{remark}
        \begin{enumerate}
            \item Global well-posedness of (\ref{eq_primitive}) in $C_t C_H L^1_{x_3}(\Torus^3 \times [0, T))$ have been established by Giga {\it et al.} \cite{GigaGriesHieberHusseinKashiwabara2017_L_infty_L1}.
            Although they consider in $\Real^2 \times \Torus$, we can use the result to derive the well-posedness in $\Torus^3$ by periodic extension.
            In \cite{GigaGriesHieberHusseinKashiwabara2017_L_infty_L1} they only consider the critical case $L^\infty_H L^1 (\Torus^3)$.
            This result can be extended to $L^\infty_H L^q (\Torus^3)$ for $q \geq 1$ using the same way in their proof.


            \item It is impossible to see inclusion for initial data between Theorem \ref{thm_main} and the result in \cite{FurukawaGigaHieberHusseinKashiwabaraWrona2018} directly since the spaces are quite different.
            We assumed less and more regularity with respect to the vertical and horizontal direction, respectively, compared to \cite{FurukawaGigaHieberHusseinKashiwabaraWrona2018}.

            \item It worth mentioning why we need not assume regularity for $v_0$ with respect to $x_3$.
            We see from the divergence-free condition that $\tilde{F} (v, w)$ is replaced by
            \begin{align*}
                -\tilde{F} (v, w)
                = F (v, w)
                + v \cdot \nabla w
                - \int_{- \pi}^{x_3}
                    \mathrm{div}_H \, v
                dz \, \mathrm{div}_H \, v.
            \end{align*}
            In this formula no $\partial_3$ appears, thus we need not control vertical derivative of $u$.
        \end{enumerate}
    \end{remark}

    The global well-posedness to the primitive equations in $L^\infty_H L^q_{x_3}(\Torus^3)$, Theorem \ref{thm_main}, and the definition of $U_\varepsilon$ imply
    \begin{corollary}
        Under the same assumptions for $\varepsilon, T, u_0, \varepsilon_0$, if $\varepsilon < \varepsilon_0$, then there exists a unique solution $u_{\varepsilon} = (v_\varepsilon, w_\varepsilon) \in C_t L^\infty_H L^q_{x_3} (\Torus^3 \times (0, T))$ to (\ref{eq_scaled_Navier-Stokes}) such that
        \begin{align}
            & \sup_{0<t<T} \Vert
                v_\varepsilon
            \Vert_{ L^\infty_H L^1_{x_3} (\Torus^3)}
            + \sup_{0<t<T} t^{1/2} \Vert
                    \nabla v_\varepsilon
            \Vert_{ L^\infty_H L^1_{x_3} (\Torus^3)} \notag \\
            & + \sup_{0<t<T} \Vert
                \varepsilon w_\varepsilon
            \Vert_{L^\infty_H L^1_{x_3} (\Torus^3)}
            + \sup_{0<t<T} t^{1/2} \Vert
                \varepsilon \nabla w_\varepsilon
            \Vert_{L^\infty_H L^1_{x_3} (\Torus^3)} \notag \\
            & \leq C \varepsilon + C^\prime,
            \end{align}
        where $C, \, C^\prime > 0$ are constants independent of $\varepsilon$.
    \end{corollary}

    The proof of Theorem \ref{thm_main} based on the contraction mapping principle.
    Note that the initial data of (\ref{eq_integral_form_diff_eq}) is zero and the external force $\varepsilon \tilde{F}(v, w)$ can be small for small $\varepsilon$.
    To estimate the right-hand-side of (\ref{eq_integral_form_diff_eq}), we need $L^\infty_H L^q_{x_3} (\Torus^3)$ bound for the composite operator $e^{t \Delta} \mathbb{P}_\varepsilon \partial_j$ for $j = 1, 2, 3$.
    Since the Riesz operator is not bounded in $L^\infty_H L^q_{x_3} (\Torus^3)$, we estimate $e^{t \Delta} \mathbb{P}_\varepsilon \partial_j$ by direct calculations for their kernel.
    The formula (\ref{eq_symbol_anisotropic_Helmholtz_projection_2}) is a key observation.
    To estimate $\tilde{F}(v, w)$ we need a control of $\nabla_H v$ since $\tilde{F}(v, w)$ has a term such as $\nabla_H v_j \nabla_H v_k$ for $j=1,2.3$ and $k=1,2$.
    In the semi-group approach, such kind of term cannot be estimated without additional regularity assumptions.
    Note that if $\nabla_H v_0 \in L^\infty_H L^q_{x_3} (\Torus^3)$ then $\nabla_H v (t) \in L^\infty_H L^q_{x_3} (\Torus^3)$ for $t \in [0, T)$ by proof of \cite{GigaGriesHieberHusseinKashiwabara2017_L_infty_L1} and the integral formulation for the primitive equations.
    For this reason we assume $\nabla_H v_0 \in L^\infty_H L^q_{x_3} (\Torus^3)$.

    To prove our main theorem, we first show short time well-posedness to (\ref{eq_difference}) and obtain the estimate (\ref{eq_Fujita_Kato_estimate}) in a short interval $[0, T_0]$ for $T_0 >0$.
    We can extended this solution to the interval $[0, 2 T_0)$ since $\Vert (V_\varepsilon (T_0), \varepsilon W_\varepsilon (T_0)) \Vert_{L^\infty_H L^q_{x_3} (\Torus^3)}$ is also small.
    Since $T$ is finite, we can establish global well-posedness to (\ref{eq_difference}) for small $\varepsilon$.
    This argument is used in \cite{FurukawaGigaHieberHusseinKashiwabaraWrona2018} and \cite{FurukawaGigaKashiwabara2021}.

    We introduce some notation in this paper.
    For $1 \leq q \leq \infty$, a domain $\Omega$, and $x \in \Omega$, we write $L^q (\Omega) = L^q_x (\Omega)$ to denote the Lebesgue space equipped with the norm
    \begin{align*}
        \Vert
            f
        \Vert_{L^q (\Omega)}
        = \Vert
            f
        \Vert_{L^q_x (\Omega)}
        : = \left(
            \int_\Omega
                \left \vert
                    f (x)
                \right \vert^q
            d x
        \right)^{1/q}.
    \end{align*}
    We use the usual modification when $q = \infty$.
    For $1 \leq q, r \leq \infty$, domains $\Omega$ and $\Omega^\prime$, and $(x^\prime, x) \in \Omega^\prime \times \Omega$, we write $L^q_{x^\prime} L^r_{x} (\Omega^\prime \times \Omega)$ to denote anisotropic Lebesgue spaces equipped with the norm
    \begin{align*}
        \Vert
            f
        \Vert_{L^q_{x^\prime} L^r_{x} (\Omega^\prime \times \Omega)}
        : = \left(
            \int_{\Omega^\prime}
                \left \Vert
                    f (x^\prime, \cdot)
                \right \Vert_{L^r(\Omega)}^q
            d x^\prime
        \right)^{1/q}.
    \end{align*}
    If $x^\prime$ and $x_3$ are the horizontal and vertical variable, respectively, then we write $L^q_H L^q_{x_3} (\Torus^3)$ to denote $L^q_{x^\prime} L^q_{x_3} (\Torus^2 \times \Torus)$.
    For vector fields $f$ and $g$, we denote their tensor product by $f \otimes g = (f_i g_j)_{ij}$.
    For an integrable function $f$ on $\Torus^3$, we denote its vertical average by $\overline{f} = \frac{1}{2} \int_{\Torus} f (x^\prime, z) dz$.
    We write $\Fourier f = \int_{\Real^d} e^{- i x \cdot \xi} f (x) d x / (2 \pi)^{d/2}$ and $\FourierInverse f = \int_{\Real^d} e^{i x \cdot \xi} f (\xi) d \xi / (2 \pi)^{d/2}$ to denote the Fourier and Fourier inverse transform for a integrable function $f$, respectively.
    We denote by $\Delta_H = \partial_1^2 + \partial_2^2$ the horizontal Laplace operator.

    This paper is organized as follows.
    In Section 2, we show linear estimate for the heat semi-group $e^{t \Delta}$ and the composite operator $e^{t \Delta} \mathbb{P}_\varepsilon \partial_j$.
    We also show some estimates for composite operators having fractional derivatives.
    In Section 3 and Section 4, we show non-linear estimates and the external force $\tilde{F}$, respectively, from the linear estimates of Section 2.
    In Section 5, we prove our main theorem.


\section{Linear Estimates}

    We give $L^\infty_H L^p_{x_3}$-$L^\infty_H L^q_{x_3}$-estimate for the hear semi-group.
    The reader refers to Section 4 of Grafakos's book \cite{Grafakos2008} for properties of the heat semi-group on $\Torus^d$.
    Let $K_t$ be the heat kernel on $\Torus^d$ for $d\geq 1$ such that
    \begin{align*}
        K_t (x)
        = \sum_{k \in \Integer^d} g_t (x - k), \quad
        x \in \Torus^d,
    \end{align*}
    where $g$ is the Gaussian of the form
    \begin{align*}
        g_t (x)
        = \frac{1}{
            \left(
                4 \pi t
            \right)^{\frac{d}{2}}
        } e^{ - \frac{
            \vert
                x
            \vert^2
            }{
                4 t
            }
        }, \quad
        x \in \Real^d.
    \end{align*}
    For a integrable function $f$, we denote by $e^{t \Delta} f = K_t \ast f$ the heat semi-group on $\Torus^d$.
    It is known that
    \begin{align*}
        & \Vert
            K_t
        \Vert_{L^1(\Torus^d)}
        = 1, \quad
        \Vert
            K_t
        \Vert_{L^\infty(\Torus^d)}
        \leq C \left(
            1 + \frac{1}{\sqrt{t}}
        \right)^d, \\
        & \Vert
            \partial^\alpha_x K_t
        \Vert_{L^1(\Torus^d)}
        \leq C t^{
            - \frac{\vert \alpha \vert}{2}
        }, \quad
        \Vert
            \partial^\alpha_x K_t
        \Vert_{L^\infty(\Torus^d)}
        \leq C t^{
            - \frac{\vert \alpha \vert}{2}
            - \frac{d}{2}
        },
    \end{align*}
    for any multi-index $\alpha$ and some constant $C>0$, see \cite{Grafakos2008} for the proof.
    We find from interpolation inequalities that
    \begin{align} \label{eq_L_p_estimate_heat_kernel_Torus}
        \begin{split}
            & \Vert
                K_t
            \Vert_{L^p(\Torus^d)}
            \leq C \left(
                1 + \frac{1}{\sqrt{t}}
            \right)^{
                d \left(
                    1 - \frac{1}{p}
                \right)
            }, \\
            & \Vert
                \partial_x^\alpha K_t
            \Vert_{L^p(\Torus^d)}
            \leq C t^{
                - \frac{\vert \alpha \vert}{2}
                - \frac{d}{2} \left(
                    1 - \frac{1}{p}
                \right)
            },
        \end{split}
    \end{align}
    for all $1 \leq p \leq \infty$.

    We next consider the composite operator with fraction derivative $(- \Delta)^{s/2} e^{t \Delta} f = M_t \ast f$ for $s >0$ and a integrable function $f$ on $\Torus^d$.
    We write $\widetilde{M}_t$ to denote the kernels of the corresponding composite operator on $\Real^d$, respectively, namely
    \begin{align*}
        \widetilde{M}_t
        = \FourierInverse |\xi|^s e^{- t |\xi|^2}, \quad \xi \in \Real^d.
   \end{align*}
    We know from Proposition 4.2 of Giga {\it et al.} \cite{GigaGriesHieberHusseinKashiwabara2017_L_infty_L1} that
    \begin{align*}
        & \Vert \widetilde{M}_t \Vert_{L^1 (\Real^d)}
        \leq C t^{- \frac{s}{2}}, \quad
        \Vert \widetilde{M}_t \Vert_{L^\infty (\Real^d)}
        \leq C t^{- \frac{s}{2} - d} \\
    \end{align*}
    for some constant $C>0$.
    Thus we obtain
    \begin{align*}
        \Vert
            \widetilde{M}_t
        \Vert_{L^p (\Real^d)}
        \leq C t^{- \frac{s}{2} - \frac{d}{2} \left(
            1 - \frac{1}{p}
        \right)}
    \end{align*}
    for all $1 \leq p \leq \infty$.
    The above observations and the Poisson summation formula yield
    \begin{align} \label{eq_bound_kernel_M_1}
        \begin{split}
            \Vert
                M_t
            \Vert_{L^p (\Torus^d)}
            & = \left \Vert
                \frac{1}{(2 \pi)^d} \sum_{n \in \Integer^d}
                |n|^{s/2} e^{i \xi \cdot n} e^{- t |n|^2}
            \right \Vert_{L^p (\Torus^d)} \\
            &= \Vert
                \widetilde{M}_t
            \Vert_{L^p(\Real^d)} \\
            & \leq C t^{- \frac{s}{2} - \frac{d}{2} \left(
                1 - \frac{1}{p}
            \right)}.
        \end{split}
    \end{align}
    The same way as above we define the horizontal and vertical composite operators $(- \Delta_H)^{s/2} e^{t \Delta_H}$ and $\partial^s_3 e^{t \partial_3^2}$.
    It worth mentioning that, in \cite{GigaGriesHieberHusseinKashiwabara2017_L_infty_L1}, they define the vertical operator using the Caputo derivative since they consider the well-posedness in the anisotropic domain $\Real^2 \times \Torus$.
    However, we do not use the Caputo derivative since we consider in $\Torus^3$.

    \begin{proposition} \label{prop_estimates_for_heat_semigroup}
        Let $1 \leq p \leq q \leq \infty$.
        Let $\gamma = (\alpha, \beta) \in \Integer^2 \times \Integer$ ($\gamma \neq 0$) be a multi-index and $0 < s_1, s_2 < 1$.
        Then there exists a constant $C>0$ such that
        \begin{align} \label{eq_L_q_L_p_estimate_heat_kernel_Torus}
            & \Vert
                e^{t \Delta} f
            \Vert_{L^\infty_H L^q_{x_3}  (\Torus^3)}
            \leq C \left(
                1 + \frac{1}{\sqrt{t}}
            \right)^{- \frac{1}{2}
                \left(
                    \frac{1}{p} - \frac{1}{q}
                \right)
            } \Vert
                    f
            \Vert_{ L^\infty_H L^p_{x_3} (\Torus^3)}, \\ \label{eq_L_q_L_p_estimate_heat_kernel_Torus_2}
            & \Vert
                \partial_{x^\prime}^{\alpha} \partial_3^{\beta} e^{t_1 \Delta_H} e^{t_2 \partial_3^2} f
            \Vert_{L^\infty_H L^q_{x_3} (\Torus^3)}
            \leq C t_1^{ - \frac{\vert \alpha \vert}{2} } t_2^{
                - \frac{\vert \beta \vert}{2} - \frac{1}{2}
                    \left(
                        \frac{1}{p} - \frac{1}{q}
                    \right)
                }
            \Vert
                f
            \Vert_{L^\infty_H L^p_{x_3}  (\Torus^3)}, \\ \label{eq_estimate_fractional_heat_semigroup}
            \begin{split}
                & \Vert
                    \partial_{x^\prime}^{\alpha} \partial_3^{\beta} (- \Delta_H)^{s_1/2} \partial_3^{s_2} e^{t_1 \Delta_H} e^{t_2 \partial_3^2} f
                \Vert_{L^\infty_H L^q_{x_3} (\Torus^3)} \\
                & \leq C t_1^{ - \frac{\vert \alpha \vert + s_1}{2} } t_2^{
                - \frac{\beta + s_2}{2} - \frac{1}{2}
                    \left(
                        \frac{1}{p} - \frac{1}{q}
                    \right)
                }
                \Vert
                    f
                \Vert_{L^\infty_H L^p_{x_3}  (\Torus^3)},
            \end{split}
        \end{align}
        for all $t, t_1, t_2 >0$ and $f \in L^\infty_H L^p_{x_3}  (\Torus^3)$.
    \end{proposition}

    \begin{proof}
        The Young inequality and (\ref{eq_L_p_estimate_heat_kernel_Torus}) imply
        \begin{align*}
            & \left \Vert
                \partial_{x^\prime}^{\alpha} \partial_3^{\beta}  e^{t_1 \Delta_H} e^{t_2 \partial_3^2} f (x^\prime, x_3)
            \right \Vert_{L^q_{x_3}(\Torus)} \\
            & \leq \int_{\Torus^2} \left|
                     \partial_{x^\prime}^{\alpha} K_{t_1}(x^\prime - y^\prime)
                 \right|
                 \left \Vert
                     \int_{\Torus}
                        \vert
                            \partial_3^{\beta} K_{t_2} (x_3- y_3)
                        \vert
                            f (y^\prime, y_3)
                        \vert
                    d y_3
                \right \Vert_{L^q_{x_3}(\Torus)}
            d y^\prime \\
            & \leq C t_2^{
                - \frac{\vert \beta \vert}{2}
                - \frac{1}{2} \left(
                    \frac{1}{p} - \frac{1}{q}
                \right)
            } \int_{\Torus^2} \left|
                     \partial_{x^\prime}^{\alpha} K_{t_1}(x^\prime - y^\prime)
                \right|
                \Vert
                    f (y^\prime, \cdot)
                \Vert_{L^p_{x_3}(\Torus)}
            d y^\prime \\
            & \leq C t_1^{ - \frac{\vert \alpha \vert}{2} } t_2^{ - \frac{\vert \beta \vert}{2} - \frac{1}{2}
            \left(
                \frac{1}{p} - \frac{1}{q}
            \right)
            } \Vert
                    f
            \Vert_{ L^\infty_H L^p_{x_3}  (\Torus^3)}.
        \end{align*}
        In view of (\ref{eq_L_p_estimate_heat_kernel_Torus}) and (\ref{eq_bound_kernel_M_1}), the inequalities (\ref{eq_L_q_L_p_estimate_heat_kernel_Torus}) and (\ref{eq_estimate_fractional_heat_semigroup}) can be proved by the same way as above.
    \end{proof}

    Let $\mathbb{P}_\varepsilon$ be the anisotropic Helmholtz projection on $\Torus^3$ with the matrix-valued symbol
    \begin{align*}
            \sigma (\mathbb{P}_\varepsilon)
            = I_3
                - \frac{1}{|\xi_\varepsilon|^2} \left(
                    \begin{array}{c}
                        \xi_1 \\
                        \xi_2 \\
                        \xi_3 / \varepsilon
                    \end{array}
                \right)
                \otimes \left(
                    \begin{array}{c}
                        \xi_1 \\
                        \xi_2 \\
                        \xi_3 / \varepsilon
                    \end{array}
                \right),
                \quad \xi \in \mathbb{Z}^3,
        \end{align*}
    where $\xi_\varepsilon = (\xi^\prime, \xi_3 / \varepsilon)$ and $\sigma (A)$ denotes the symbol of a multiplier operator $A$.
    The projection $\mathbb{P}_\varepsilon$ is an unbounded operator on $L^\infty_H L^q_{x_3} (\Torus^3)$ since the Riesz operator is unbounded on $L^\infty(\Torus^d)$ for all $d \geq 1$.
    However, the composite operators $e^{t \Delta} \mathbb{P}_\varepsilon \partial_j$ for $j = 1, 2, 3$ and $t>0$ are bounded on $L^\infty_H L^p_{x_3} (\Torus^3)$.
    We can rewrite the anisotropic Helmholtz projection as
    \begin{align} \label{eq_symbol_anisotropic_Helmholtz_projection_2}
        \sigma (\mathbb{P}_\varepsilon)
        = \left(
            \begin{array}{cc}
                I_2
                & 0 \\
                0
                & 0
            \end{array}
        \right)
        - \frac{1}{
            \vert
                \xi_\varepsilon
            \vert^2
        } \left(
            \begin{array}{cc}
                \xi^\prime \otimes \xi^\prime
                & \xi^\prime \xi_3 / \varepsilon \\
                {\xi^\prime}^T \xi_3 / \varepsilon
                & - \vert \xi^\prime \vert^2
            \end{array}
        \right).
    \end{align}
    This is a key formula to show the boundedness for $e^{t \Delta} \mathbb{P}_\varepsilon \partial_j$.
    We denote by $R_j$ the anisotropic Riesz operator with symbol $\xi_j / |\xi_\varepsilon|$ for $j = 1, 2, 3$.
    We write $R^\prime = (R_1, R_2)^T$.

    We prove elementally estimates, which is used to show $\varepsilon$-independent bounds for composite operators.
    It may be somewhat prolix, but we show calculation to clarify dependence of $\varepsilon$ for the estimates.
    \begin{proposition} \label{prop_estimate_fational_function_epsilon_independent}
        Let $0 < \alpha \leq 1$ and $0 < \beta < 1$.
        Then there exists a constant $C>0$ such that
        \begin{align}
            & \int_0^\infty
                (t + s)^{- 1 - \frac{\alpha}{2}} \left(
                    t + \frac{s}{\varepsilon^2}
                \right)^{
                    - \frac{\beta}{2}
                }
            ds
            \leq C t^{
                - \frac{\alpha}{2} - \frac{\beta}{2}
            } \varepsilon^{\beta}, \label{eq_estimate_fational_function_epsilon_independent_1} \\
            & \int_0^\infty
                (t + s)^{- \frac{1}{2} - \frac{\alpha}{2}} \left(
                    t + \frac{s}{\varepsilon^2}
                \right)^{
                    - \frac{1}{2}
                    - \frac{\beta}{2}
                }
            ds 
            \leq C t^{
                - \frac{\alpha}{2} - \frac{\beta}{2}
            } \varepsilon^{1 + \beta}, \label{eq_estimate_fational_function_epsilon_independent_2} \\
            & \int_0^\infty
                (t + s)^{-1} \left(
                    t + \frac{s}{\varepsilon^2}
                \right)^{
                    -\frac{\alpha}{2} - \frac{\beta}{2}
                }
            ds
            \leq C t^{
                - \frac{\alpha}{2} - \frac{\beta}{2}
            } \varepsilon^{
                \alpha + \beta
            }, \label{eq_estimate_fational_function_epsilon_independent_3}\\
            & \int_0^\infty
                (t + s)^{- \frac{1}{2}} \left(
                    t + \frac{s}{\varepsilon^2}
                \right)^{
                    - \frac{1}{2} - \frac{\alpha}{2} - \frac{\beta}{2}
                }
            ds
            \leq C t^{
                - \frac{\alpha}{2} - \frac{\beta}{2}
            } \varepsilon. \label{eq_estimate_fational_function_epsilon_independent_4}
        \end{align}
        for all $t>0$ and $0 < \varepsilon \leq 1$.
    \end{proposition}
    \begin{proof}
        We first prove (\ref{eq_estimate_fational_function_epsilon_independent_1}).
        The change of variable $s = t s^\prime$ and the inequality $(\varepsilon^2 + s)^{-\beta/2} \leq s^{- \beta/2}$ yield
        \begin{align*}
            \int_0^\infty
                (t + s)^{- 1 - \frac{\alpha}{2}} \left(
                    t + \frac{s}{\varepsilon^2}
                \right)^{
                    - \frac{\beta}{2}
                }
            ds
            & \leq C t^{
                - \frac{\alpha}{2} - \frac{\beta}{2}
            } \varepsilon^{\beta} \int_0^\infty
                (1 + s)^{- 1 - \frac{\alpha}{2}} s^{ - \beta
                }
            ds \\
            & \leq C t^{
                - \frac{\alpha}{2} - \frac{\beta}{2}
            } \varepsilon^{\beta}.
        \end{align*}
        For (\ref{eq_estimate_fational_function_epsilon_independent_2}) we divide the integral interval to see that
        \begin{align}
                & \int_0^\infty
                    (t + s)^{- \frac{1}{2} - \frac{\alpha}{2}} \left(
                        t + \frac{s}{\varepsilon^2}
                    \right)^{
                        - \frac{1}{2}
                        - \frac{\beta}{2}
                    }
                ds \notag \\
                & \leq C t^{
                    - \frac{\alpha}{2}
                    - \frac{\beta}{2}
                } \varepsilon^{1 + \beta}
                \int_0^1
                        (1 + s)^{- \frac{1}{2} - \frac{\alpha}{2}} (\varepsilon^2 + s)^{- \frac{1}{2} - \frac{\beta}{2}}
                ds \notag \\
                & + C t^{
                    - \frac{\alpha}{2} - \frac{\beta}{2}
                } \varepsilon^{1 + \beta}
                \int_1^\infty
                        (1 + s)^{- \frac{1}{2} - \frac{\alpha}{2}} (\varepsilon^2 + s)^{- \frac{1}{2} - \frac{\beta}{2}}
                ds \notag \\
                & \leq C t^{
                    - \frac{\alpha}{2} - \frac{\beta}{2}
                } \varepsilon^{1 + \beta}
                \int_0^1
                        s^{- \frac{1}{2} - \frac{\beta}{2}}
                ds \notag \\
                & + C t^{
                    - \frac{\alpha}{2} - \frac{\beta}{2}
                } \varepsilon^{1 + \beta}
                \int_1^\infty
                        (1 + s)^{- \frac{1}{2} - \frac{\alpha}{2}} s^{- \frac{1}{2} - \frac{\beta}{2}}
                ds \notag \\ 
                &\leq C t^{
                    - \frac{\alpha}{2} - \frac{\beta}{2}
                } \varepsilon^{1 + \beta}.
        \end{align}
        Similar to the first inequality, we use the change of variable $s = t s^\prime$ to estimate
        \begin{align*}
            \frac{1}{\varepsilon} \int_0^\infty
                (t + s)^{- 1} \left(
                    t + \frac{s}{\varepsilon^2}
                \right)^{
                    - \frac{\alpha}{2} - \frac{\beta}{2}
                }
            ds
            & \leq C t^{
                - \frac{\alpha}{2} - \frac{\beta}{2}
            } \varepsilon^{
                \alpha + \beta
            }
            \int_0^\infty
                    (1 + s)^{-1} (\varepsilon^2 + s)^{ - \frac{\alpha}{2} - \frac{\beta}{2}}
            ds \\
            & \leq C t^{
                - \frac{\alpha}{2} - \frac{\beta}{2}
            } \varepsilon^{
                \alpha + \beta
            }
            \int_0^\infty
                    (1 + s)^{- 1} s^{ - \frac{\alpha}{2} + \frac{\beta}{2}}
            ds \\
            & \leq C t^{
                - \frac{\alpha}{2} - \frac{\beta}{2}
            } \varepsilon^{
                \alpha + \beta
            }.
        \end{align*}
        We proved (\ref{eq_estimate_fational_function_epsilon_independent_3}).
        For the last inequality, we apply the change of variables $s = \varepsilon^2 t s^\prime$ and the estimate $\varepsilon / (1 + \varepsilon^2 s)^{1/2} \leq 1 / s^{1/2}$ to get
        \begin{align*}
            & \int_0^\infty
                (t + s)^{- \frac{1}{2}} \left(
                    t + \frac{s}{\varepsilon^2}
                \right)^{
                    - \frac{1}{2} - \frac{\alpha}{2} - \frac{\beta}{2}
                }
            ds \\
            & \leq C t^{
                - \frac{\alpha}{2} - \frac{\beta}{2}
            } \varepsilon^2
            \int_0^\infty
                    (1 + \varepsilon^2 s)^{- \frac{1}{2}} (1 + s)^{ - \frac{1}{2} - \frac{\alpha}{2} - \frac{\beta}{2}}
            ds \\
            & \leq C t^{
                - \frac{\alpha}{2} - \frac{\beta}{2}
            } \varepsilon
            \int_0^\infty
                    s^{- \frac{1}{2}} (1 + s)^{ - \frac{1}{2} - \frac{\alpha}{2} - \frac{\beta}{2}}
            ds \\
            & \leq C t^{
                - \frac{\alpha}{2} - \frac{\beta}{2}
            } \varepsilon.
        \end{align*}
        We obtain (\ref{eq_estimate_fational_function_epsilon_independent_4}).
    \end{proof}

    \begin{proposition}\label{prop_estimate_composit_operator}
        Let $1 \leq p, q \leq \infty$, and $0 < s < 1$.
        Then there exists a constant $C>0$ such that
        \begin{align}
            \Vert
                e^{t \Delta} \mathbb{P}_\varepsilon \partial_j f
            \Vert_{L^\infty_H L^q_{x_3} (\Torus^3)}
            & \leq C t^{
                - \frac{1}{2}
                - \frac{1}{2} \left(
                    \frac{1}{p} - \frac{1}{q}
                \right)
            }
            \Vert
                f
            \Vert_{L^\infty_H L^p_{x_3} (\Torus^3)}, \\
            \Vert
                e^{t \Delta} \mathbb{P}_\varepsilon (- \Delta_H)^{s/2} f
            \Vert_{L^\infty_H L^q_{x_3} (\Torus^3)}
            & \leq C t^{
                - \frac{s}{2}
                - \frac{1}{2} \left(
                    \frac{1}{p} - \frac{1}{q}
                \right)
            }
            \Vert
                f
            \Vert_{L^\infty_H L^p_{x_3} (\Torus^3)}, \\
            \Vert
                e^{t \Delta} \mathbb{P}_\varepsilon \partial_3^s f
            \Vert_{L^\infty_H L^q_{x_3} (\Torus^3)}
            & \leq C t^{
                - \frac{s}{2}
                - \frac{1}{2} \left(
                    \frac{1}{p} - \frac{1}{q}
                \right)
            }
            \Vert
                f
            \Vert_{L^\infty_H L^p_{x_3} (\Torus^3)},
        \end{align}
        for all $0 < \varepsilon <1$, $t>0$, $f \in L^\infty_H L^q_{x_3} (\Torus^3)$, and $j = 1, 2, 3$.
    \end{proposition}

    \begin{proof}
        We prove the first inequality.
        The second and third inequalities can be proved by the completely same way combining with Propositions \ref{prop_estimates_for_heat_semigroup} and \ref{prop_estimate_fational_function_epsilon_independent}.
        The formula (\ref{eq_symbol_anisotropic_Helmholtz_projection_2}) and
        \begin{align*}
                (- \Delta)^{-\frac{\alpha}{2}} f
            = \frac{1}{\Gamma (\frac{\alpha}{2})} \int_0^\infty
                s^{\frac{\alpha}{2} - 1} (K_s \ast f)
            ds
        \end{align*}
        from Section 4 of \cite{GigaGriesHieberHusseinKashiwabara2017_L_infty_L1}, where $\alpha > 0$ and $\Gamma$ is the gamma function, lead to
        \begin{align*}
            \partial_j e^{t \Delta} \mathbb{P}_\varepsilon
            & = \partial_j e^{t \Delta} \left(
                \begin{array}{cc}
                    I_2
                    & 0 \\
                    0
                    & 0
                \end{array}
            \right)
            - \partial_j e^{t \Delta} \left(
                \begin{array}{cc}
                    R^\prime \otimes R^\prime
                    & R^\prime R_3 / \varepsilon \\
                    {R^\prime}^T R_3 / \varepsilon
                    & - R_1^2 - R_2^2
                \end{array}
            \right) \\
            & = \partial_j e^{t \Delta} \left(
                \begin{array}{cc}
                    I_2
                    & 0 \\
                    0
                    & 0
                \end{array}
            \right) \\
            & -  \int_0^\infty
                \partial_j \left(
                    \begin{array}{cc}
                        \nabla_H \otimes \nabla_H
                        & \nabla_H \partial_3 / \varepsilon \\
                        {\nabla_H}^T \partial_3 / \varepsilon
                        & - \partial_1^2 - \partial^2_2
                    \end{array}
                \right)
                e^{(t + s) \Delta_H } e^{(t + s / \varepsilon^2) \partial_3^2}
            ds.
        \end{align*}
        The operator norm of the first term from $L^\infty_H L^p_{x_3} (\Torus^3)$ to $L^\infty_H L^q_{x_3} (\Torus^3)$ is bounded by Proposition \ref{prop_estimates_for_heat_semigroup}.
        We use Proposition \ref{prop_estimate_composit_operator} to estimate
        \begin{align*}
            &\left \Vert
                \partial_j \nabla_H \otimes \nabla_H \int_0^\infty
                    e^{(t + s) \Delta_H } e^{(t + s / \varepsilon^2) \partial_3^2} f
                ds
            \right \Vert_{L^\infty_H L^q_{x_3}  (\Torus^3)} \\
            & \leq C \Vert
                f
            \Vert_{L^\infty_H L^p_{x_3}  (\Torus^3)}
            \left \{
                \begin{array}{c}
                    \int_0^\infty
                        (t + s)^{- \frac{3}{2}} (t + \frac{s}{\varepsilon^2})^{
                            - \frac{1}{2} \left(
                                \frac{1}{p} - \frac{1}{q}
                            \right)
                        }
                    ds, \quad j = 1,2, \\
                    \int_0^\infty
                        (t + s)^{- 1} (t + \frac{s}{\varepsilon^2})^{
                            - \frac{1}{2} - \frac{1}{2} \left(
                                \frac{1}{p    } - \frac{1}{q}
                            \right)
                        }
                    ds, \quad j = 3.
                \end{array}
            \right.
        \end{align*}
        We use Proposition \ref{prop_estimate_fational_function_epsilon_independent} to see that
        \begin{align} \label{eq_sec2_prop_estimate_composit_operator_15}
            \begin{split}
                &\left \Vert
                    \partial_j \nabla_H \otimes \nabla_H \int_0^\infty
                        e^{(t + s) \Delta_H } e^{(t + s / \varepsilon^2)    \partial_3^2} f
                    ds
                \right \Vert_{L^\infty_H L^q_{x_3} (\Torus^3)} \\
                &\leq C t^{
                    - \frac{1}{2} - \frac{1}{2} \left(
                        \frac{1}{p} - \frac{1}{q}
                    \right)
                } \varepsilon^{\frac{1}{p} - \frac{1}{q}} \Vert
                    f
                \Vert_{L^\infty_H L^p_{x_3} (\Torus^3)}.
            \end{split}
        \end{align}
        Proposition \ref{prop_estimate_fational_function_epsilon_independent} leads to
        \begin{align*}
            & \left \Vert
                \nabla_H \frac{\partial_3}{\varepsilon} \partial_j \int_0^\infty
                    e^{(t + s) \Delta_H} e^{(t + s / \varepsilon^2) \partial_3^2} f
                ds
            \right \Vert_{L^\infty_H L^q_{x_3} (\Torus^3)} \\
            & \leq C \Vert
                f
            \Vert_{L^\infty_H L^p_{x_3}  (\Torus^3)} \left \{
                \begin{array}{l}
                    \frac{1}{\varepsilon} \int_0^\infty
                        (t + s)^{-1} (t + \frac{s}{\varepsilon^2})^{
                            - \frac{1}{2} - \frac{1}{2} \left(
                                \frac{1}{p} - \frac{1}{q}
                            \right)
                        }
                    ds, \quad j = 1,2, \\
                    \frac{1}{\varepsilon} \int_0^\infty
                        (t + s)^{- \frac{1}{2}} (t + \frac{s}{\varepsilon^2})^{
                            - 1 - \frac{1}{2} \left(
                                \frac{1}{p} - \frac{1}{q}
                            \right)
                        }
                    ds, \quad j = 3.
                \end{array}
            \right.
        \end{align*}
        Thus, Proposition \ref{prop_estimate_fational_function_epsilon_independent} implies
        \begin{align} \label{eq_sec2_prop_estimate_composit_operator_20}
            \left \Vert
                \nabla_H \frac{\partial_3}{\varepsilon} \partial_j \int_0^\infty
                    e^{(t + s) \Delta_H} e^{(t + s / \varepsilon^2) \partial_3^2} f
                ds
            \right \Vert_{L^\infty_H L^q_{x_3}  (\Torus^3)}
            \leq C t^{
                - \frac{1}{2} - \frac{1}{2} \left(
                    \frac{1}{p} - \frac{1}{q}
                \right)
            } \Vert
                f
            \Vert_{L^\infty_H L^p_{x_3}(\Torus^3)}.
        \end{align}
        The conclusion follows from (\ref{eq_sec2_prop_estimate_composit_operator_15}) and (\ref{eq_sec2_prop_estimate_composit_operator_20}).
        We proved Proposition \ref{prop_estimate_composit_operator}.
    \end{proof}

    \begin{remark}
        The formula (\ref{eq_symbol_anisotropic_Helmholtz_projection_2}) plays an essential role in the proof of Proposition \ref{prop_estimate_composit_operator}.
        If we try to estimate the norm of $e^{t \Delta} \mathbb{P}_\varepsilon \partial_3^2 / \varepsilon^2$, then we have to deal the term $\int_0^\infty \varepsilon^{-2} \partial_3^2 \partial_j e^{(t+s)\Delta_H} e^{(t+s/\varepsilon^2)\partial_3^2} f ds$ for $j=1,2$.
        However, it is impossible to get the uniform $L^\infty_H L^p_{x_3}$-bound on $\varepsilon$ since
        \begin{align*}
            & \left \Vert
                \int_0^\infty
                    \frac{
                        \partial_3^2
                    }{
                        \varepsilon^{2}
                    }
                    \partial_j e^{(t+s)\Delta_H} e^{(t+s/\varepsilon^2)  \partial_3^2} f
                ds
            \right \Vert_{L^\infty_H L^p_{x_3}(\Torus^3)} \\
            & \leq C \Vert
                f
            \Vert_{L^\infty_H L^p_{x_3}(\Torus^3)} \frac{1}{\varepsilon^2}
            \int_0^\infty
                \frac{
                    1
                }{
                    (t + s)^{\frac{1}{2}} (t + \frac{s}{\varepsilon^2})
                }
            ds \\
            & = C \Vert
                f
            \Vert_{L^\infty_H L^p_{x_3}(\Torus^3)}
            t^{
                - \frac{1}{2}
            }
            \int_0^\infty
                \frac{
                    1
                }{
                    (1 + s)^{\frac{1}{2}} (\varepsilon^2 + s)
                }
            ds,
        \end{align*}
        for a $\varepsilon$-independent constant $C>0$.
        If $\varepsilon$ tends to zero, then the integral of the right-hand-side goes to infinity as $| \log \varepsilon |$.
        Thus, in this calculations, we lose uniform estimates on $\varepsilon$.
    \end{remark}

    \begin{corollary}\label{cor_estimate_composit_operator}
        Let $1 \leq p \leq \infty$ and $T>0$.
        Then there exists a constant $C>0$ such that
        \begin{align*}
            \sup_{0< t < T} \left \Vert
                \int_0^t
                    e^{(t-s) \Delta} \mathbb{P}_\varepsilon \partial_j f (s)
                ds
            \right \Vert_{L^\infty_H L^p_{x_3} (\Torus^3)}
            & \leq C t^{1/2}
            \sup_{0< t < T} \Vert
                f (t)
            \Vert_{L^\infty_H L^p_{x_3} (\Torus^3)}
        \end{align*}
        for all $f \in L^\infty_t L^\infty_H L^p_{x_3} (\Torus^3 \times (0, T))$, $0 < \varepsilon \leq 1$, and $j = 1, 2, 3$.
    \end{corollary}

\section{Non-Linear Estimates in $L^\infty_H L^p_{x_3}  (\Torus^3)$}
    In this section, we show some non-linear estimates for some quadratic terms.
\subsection{Non-linear estimates for composite operators}
    The following proposition is elemental but is very useful.
    \begin{proposition} \label{prop_L_infty_L_p_estimate}
        Let $1 \leq p \leq \infty$.
        There exists a constant $C>0$ such that
        \begin{align*}
            \Vert
                f
            \Vert_{L^\infty (\Torus^3)}
            \leq \Vert
                f
            \Vert_{L^\infty_H L^p_{x_3}  (\Torus^3)}
            + \Vert
                \partial_3 f
            \Vert_{L^\infty_H L^p_{x_3}  (\Torus^3)},
        \end{align*}
        for all $f \in L^\infty_H L^p_{x_3}  (\Torus^3)$ satisfying $\partial_3 f \in L^\infty_H L^p_{x_3}  (\Torus^3)$.
        In particular, if $\overline{f} = 0$, then
        \begin{align} \label{eq_estimate_Linfty_x_by_Linfty_H-L1_3}
            \Vert
                f
            \Vert_{L^\infty (\Torus^3)}
            \leq \Vert
                \partial_3 f
            \Vert_{L^\infty_H L^p_{x_3}  (\Torus^3)}.
        \end{align}
    \end{proposition}
    \begin{proof}
        By the fundamental theorem of calculus, we have the pointwise estimate
        \begin{align*}
            \left|
                f(\cdot, x_3)
                - \overline{f} (\cdot)
            \right|
            \leq \int_{- \pi}^{\pi}
                | \partial_z f(\cdot, z) |
            dz,
        \end{align*}
        for all $\pi \leq x_3 \leq \pi$.
        Applying $L^\infty_H L^p_{x_3}$-norm to the both sides and using the triangle inequality, we have (\ref{eq_estimate_Linfty_x_by_Linfty_H-L1_3}).
    \end{proof}

    To estimate $F_H (U_\varepsilon, u)$ and $\varepsilon F_3 (U_\varepsilon, u)$ in (\ref{eq_def_of_F}), we show
    \begin{proposition} \label{prop_nonlinear_estimate_div_ep_free}
        Let $1 \leq p \leq \infty$.
        Then there exists a constant $C>0$ such that
        \begin{align} \label{eq_nonlinear_estimate_div_ep_free}
            \begin{split}
                & \Vert
                    e^{t \Delta} \mathbb{P}_\varepsilon \mathrm{div}_\varepsilon \left(
                        f \otimes g
                    \right)
                \Vert_{L^\infty_{H} L^p_{x_3}} \\
                & \leq C t^{- 1 / 2} \min \left[
                        \Vert
                        f
                    \Vert_{L^\infty_H L^p_{x_3}  (\Torus^3)}
                    (
                        \Vert
                            g
                        \Vert_{L^\infty_H L^p_{x_3}  (\Torus^3)}
                        + \Vert
                            \nabla g
                        \Vert_{L^\infty_H L^p_{x_3}  (\Torus^3)}
                    ),
                \right. \\
                &\left.
                    \left(
                        \Vert
                            f
                        \Vert_{L^\infty_H L^p_{x_3} (\Torus^3)}
                        + \Vert
                            \nabla f
                        \Vert_{L^\infty_H L^p_{x_3} (\Torus^3)}
                    \right)
                    \Vert
                        g
                    \Vert_{L^\infty_H L^p_{x_3}(\Torus^3)}
                \right] \\
                & + C t^{ - 1 / 2} \Vert
                    f
                \Vert_{L^\infty_H L^p_{x_3}  (\Torus^3)}
                \Vert
                    \nabla g
                \Vert_{L^\infty_H L^p_{x_3} (\Torus^3)},
            \end{split}
        \end{align}
        for all $t>0$, $\mathrm{div}_\varepsilon$-free vector fields $f, g \in L^\infty_H L^p_{x_3} (\Torus^3)$ satisfying $\nabla f, \nabla g \in L^\infty_H L^p_{x_3} (\Torus^3)$, and $0 < \varepsilon \leq 1$.
    \end{proposition}
    \begin{proof}
        We use Proposition \ref{prop_estimate_composit_operator} and the formula
        \begin{align} \label{eq_tensor_product_division}
            e^{t \Delta} \mathbb{P}_\varepsilon \mathrm{div}_\varepsilon (f \otimes g)
            = e^{t \Delta} \mathbb{P}_\varepsilon \mathrm{div}_H \left(
                f g_H^T
            \right)
            + e^{t \Delta} \mathbb{P}_\varepsilon \frac{\partial_3}{\varepsilon} \left(
                f g_3
            \right),
        \end{align}
        to see
        \begin{align*}
            & \left \Vert
                e^{t \Delta} \mathbb{P}_\varepsilon \mathrm{div}_\varepsilon (f \otimes g)
            \right \Vert_{L^\infty_H L^p_{x_3}  (\Torus^3)} \\
            & \leq C t^{ - 1 / 2} \Vert
                f g_H^T
            \Vert_{L^\infty_H L^p_{x_3}  (\Torus^3)}
            + C \frac{t^{ - \frac{1}{2}}}{\varepsilon} \Vert
                f g_3
            \Vert_{L^\infty_H L^p_{x_3}  (\Torus^3)} \\
            & = : I_1 + I_2.
        \end{align*}
        Proposition \ref{prop_L_infty_L_p_estimate} implies
        \begin{align*}
            I_1
            & \leq C t^{ - \frac{1}{2}} \Vert
                f
            \Vert_{L^\infty_H L^p_{x_3}  (\Torus^3)}
            \left(
                \Vert
                    g_H
                \Vert_{L^\infty_H L^p_{x_3}  (\Torus^3)}
                + \Vert
                    \partial_3 g_H
                \Vert_{L^\infty_H L^p_{x_3}  (\Torus^3)}
            \right),
        \end{align*}
        and
        \begin{align*}
            I_1
            & \leq C t^{ - \frac{1}{2}} \left( \Vert
                    f
                \Vert_{L^\infty_H L^p_{x_3}  (\Torus^3)}
                + \Vert
                    \partial_3 f
                \Vert_{L^\infty_H L^p_{x_3}  (\Torus^3)}
                \right)
                \Vert
                    g_H
                \Vert_{L^\infty_H L^p_{x_3}  (\Torus^3)}.
        \end{align*}
        Since the $\mathrm{div}_\varepsilon$-free condition yields the formula
        \begin{align} \label{eq_estimate_g3_by_nablagH}
            g_3
            = - \varepsilon \int_{- \pi}^{x_3}
                \mathrm{div}_H \, g_H
            dz, \quad
            x_3 \in \Torus,
        \end{align}
        we see the factor $\varepsilon^{-1}$ in $I_2$ is canceled and estimate $I_2$ as
        \begin{align*}
            I_2
            & \leq C t^{ - \frac{1}{2}} \Vert
                f
            \Vert_{L^\infty_H L^p_{x_3}  (\Torus^3)}
            \Vert
                \nabla_H g_H
            \Vert_{L^\infty_H L^p_{x_3} (\Torus^3)}.
        \end{align*}
        Thus we obtained (\ref{eq_nonlinear_estimate_div_ep_free}).
    \end{proof}

    \begin{corollary} \label{cor_nonlinear_estimate_div_ep_free}
        Let $1 \leq p \leq \infty$, $T>0$.
        Then there exists a constant $C>0$ such that
        \begin{align} \label{eq_nonlinear_estimate_div_ep_free_time}
            \begin{split}
                & \sup_{0 < t < T} \left \Vert
                    \int_0^t
                        e^{(t-s) \Delta} \mathbb{P}_\varepsilon \mathrm{div}_\varepsilon \left(
                            f \otimes g
                        \right)
                    ds
                \right \Vert_{L^\infty_H L^p_{x_3} (\Torus^3)} \\
                & \leq C \min \left[
                    \sup_{0 < t < T} \Vert
                        f (t)
                    \Vert_{L^\infty_H L^p_{x_3} (\Torus^3)}
                    \left(
                        \sup_{0 < t < T} t^{1/2} \Vert
                            g (t)
                        \Vert_{L^\infty_H L^p_{x_3} (\Torus^3)}
                        + \sup_{0 < t < T} t^{1/2} \Vert
                            \nabla g (t)
                        \Vert_{L^\infty_H L^p_{x_3}  (\Torus^3)}
                    \right),
                \right. \\
                & \quad \quad \quad \quad \left.
                    \left(
                        \sup_{0 < t < T} t^{\frac{1}{2}} \Vert
                            f (t)
                        \Vert_{L^\infty_H L^p_{x_3} (\Torus^3)}
                        + \sup_{0 < t < T} t^{1/2} \Vert
                            \nabla f (t)
                        \Vert_{L^\infty_H L^p_{x_3} (\Torus^3)}
                    \right)
                    \sup_{0 < t < T} \Vert
                        g (t)
                    \Vert_{L^\infty_H L^p_{x_3} (\Torus^3)}
                \right] \\
                & + C \sup_{0 < t < T} \Vert
                    f (t)
                \Vert_{L^\infty_H L^p_{x_3} (\Torus^3)}
                \sup_{0 < t < T} t^{\frac{1}{2}} \Vert
                    \nabla g (t)
                \Vert_{L^\infty_H L^p_{x_3} (\Torus^3)}
            \end{split}
        \end{align}
        for all $t>0$, $\mathrm{div}_\varepsilon$-free vector fields $f, g \in L^\infty_t L^\infty_H L^p_{x_3} (\Torus^3 \times (0, T))$ satisfying $t^{\frac{1}{2}}\nabla f, t^{\frac{1}{2}} \nabla g \in L^\infty_t L^\infty_H L^p_{x_3} (\Torus^3 \times (0, T))$, and $0 < \varepsilon \leq 1$.
    \end{corollary}
    \begin{proof}
        Taking $L^\infty_t$-norm to both sides of (\ref{eq_nonlinear_estimate_div_ep_free}), we have (\ref{eq_nonlinear_estimate_div_ep_free_time}).
    \end{proof}

    We estimate space derivatives to the quadratic terms.
    \begin{proposition} \label{prop_estimates_nabla_nonlinear_terms}
        Let $1 \leq p \leq \infty$.
        Then there exists a constant $C>0$ such that
        \begin{align} \label{eq_estimates_nabla_nonlinear_terms1}
            \begin{split}
                & \Vert
                    \nabla e^{t \Delta} \mathbb{P}_\varepsilon \mathrm{div}_\varepsilon     \left(
                        f \otimes g
                    \right)
                \Vert_{L^\infty_H L^p_{x_3} (\Torus^3)} \\
                &\leq C t^{-1} \min \left[
                    \left(
                        \Vert
                            f
                        \Vert_{L^\infty_H L^p_{x_3} (\Torus^3)}
                        + \Vert
                            \nabla f
                        \Vert_{L^\infty_H L^p_{x_3} (\Torus^3)}
                    \right)
                    \Vert
                        g
                    \Vert_{L^\infty_H L^p_{x_3} (\Torus^3)},
                \right. \\
                & \quad \quad \quad \quad \quad \quad \quad \left.
                    \Vert
                        f
                    \Vert_{L^\infty_H L^p_{x_3} (\Torus^3)}
                    \left(
                        \Vert
                            g
                        \Vert_{L^\infty_H L^p_{x_3} (\Torus^3)}
                        + \Vert
                            \nabla g
                        \Vert_{L^\infty_H L^p_{x_3} (\Torus^3)}
                    \right)
                \right] \\
                & + C t^{-1} \Vert
                    f
                \Vert_{L^\infty_H L^p_{x_3} (\Torus^3)}
                \Vert
                    \nabla g
                \Vert_{L^\infty_H L^p_{x_3} (\Torus^3)},
            \end{split}
        \end{align}
        and
        \begin{align} \label{eq_estimates_nabla_nonlinear_terms2}
            \begin{split}
                & \Vert
                    \nabla e^{t \Delta} \mathbb{P}_\varepsilon \mathrm{div} _\varepsilon \left(
                        f \otimes g
                    \right)
                \Vert_{L^\infty_H L^p_{x_3} (\Torus^3)} \\
                & \leq C t^{- 1/2} \Vert
                    \nabla f
                \Vert_{L^\infty_H L^p_{x_3} (\Torus^3)}
                \left(
                    \Vert
                        g
                    \Vert_{L^\infty_H L^p_{x_3} (\Torus^3)}
                    + \Vert
                        \nabla g
                    \Vert_{L^\infty_H L^p_{x_3} (\Torus^3)}
                \right), 
            \end{split}
        \end{align}
        for all $\mathrm{div}_\varepsilon$-free vector fields $f, g \in L^\infty_H L^p_{x_3} (\Torus^3)$ satisfying $\nabla f, \nabla g \in L^\infty_H L^p_{x_3} (\Torus^3)$, and $0 < \varepsilon \leq 1$.
    \end{proposition}
    \begin{proof}
        We prove the first inequality.
        By the formula (\ref{eq_tensor_product_division}), we see that
        \begin{align*}
            & \Vert
                \nabla e^{t \Delta} \mathbb{P}_\varepsilon \mathrm{div}_\varepsilon \left(
                    f \otimes g
                \right)
            \Vert_{L^\infty_H L^p_{x_3} (\Torus^3)} \\
            & \leq \Vert
                \nabla e^{t \Delta} \mathbb{P}_\varepsilon \mathrm{div}_H \left(
                    f \otimes g_H
                \right)
            \Vert_{L^\infty_H L^p_{x_3} (\Torus^3)}
            + \left \Vert
                \nabla e^{t \Delta} \mathbb{P}_\varepsilon \left(
                    f \int_{- \pi}^{x_3}
                        \mathrm{div}_H \, g_H
                    dz
                \right)
            \right \Vert_{L^\infty_H L^p_{x_3} (\Torus^3)} \\
            & =: I_1 + I_2.
        \end{align*}
        Propositions \ref{prop_estimate_composit_operator} and \ref{eq_estimate_Linfty_x_by_Linfty_H-L1_3} imply
        \begin{align*}
            I_1
            & \leq C t^{-1} \Vert
                f \otimes g_H
            \Vert_{L^\infty_H L^p_{x_3} (\Torus^3)} \\
            &\leq C t^{-1} \min \left[
                \left(
                    \Vert
                        f
                    \Vert_{L^\infty_H L^p_{x_3} (\Torus^3)}
                    + \Vert
                        \partial_3 f
                    \Vert_{L^\infty_H L^p_{x_3} (\Torus^3)}
                \right)
                \Vert
                    g_H
                \Vert_{L^\infty_H L^p_{x_3} (\Torus^3)},
            \right. \\
            & \quad \quad \quad \quad \quad \quad \quad \left.
                \Vert
                    f
                \Vert_{L^\infty_H L^p_{x_3} (\Torus^3)}
                \left(
                    \Vert
                        g_H
                    \Vert_{L^\infty_H L^p_{x_3} (\Torus^3)}
                    + \Vert
                    \partial_3 g_H
                    \Vert_{L^\infty_H L^p_{x_3} (\Torus^3)}
                \right)
            \right],
        \end{align*}
        and
        \begin{align*}
            I_2
            \leq C t^{-1} \Vert
                f
            \Vert_{L^\infty_H L^p_{x_3} (\Torus^3)}
            \Vert
                \nabla_H g_H
            \Vert_{L^\infty_H L^p_{x_3} (\Torus^3)}.
        \end{align*}
        Next, we prove the last inequality.
        We see from the $\mathrm{div}_\varepsilon$-free condition that
        \begin{align*}
            \mathrm{div}_\varepsilon (f \otimes g)
            = g_H \cdot \nabla_H f
            - \int_{- \pi}^{x_3}
                \mathrm{div}_H \, g_H
            dz \, \partial_3 f.
        \end{align*}
        Thus we obtain by the same way as above that
        \begin{align*}
            & \Vert
                \nabla e^{t \Delta} \mathbb{P}_\varepsilon \mathrm{div}_\varepsilon \left(
                    f \otimes g
                \right)
            \Vert_{L^\infty_H L^p_{x_3} (\Torus^3)} \\
            & \leq C t^{-1/2} \Vert
                g_H \cdot \nabla_H f
            \Vert_{L^\infty_H L^p_{x_3} (\Torus^3)}
            + C t^{-\frac{1}{2}} \left \Vert
                \int_{- \pi}^{x_3}
                    \mathrm{div}_H \, g_H
                dz \partial_3 f
            \right \Vert_{L^\infty_H L^p_{x_3} (\Torus^3)} \\
            & \leq C t^{-\frac{1}{2}} \left(
                \Vert
                    g_H
                \Vert_{L^\infty_H L^p_{x_3} (\Torus^3)}
                + \Vert
                    \partial_3 g_H
                \Vert_{L^\infty_H L^p_{x_3} (\Torus^3)}
            \right) \Vert
                \nabla_H f
            \Vert_{L^\infty_H L^p_{x_3} (\Torus^3)} \\
            & + C t^{-\frac{1}{2}} \Vert
                \nabla_H g_H
            \Vert_{L^\infty_H L^p_{x_3} (\Torus^3)}
            \Vert
                \partial_3 f
            \Vert_{L^\infty_H L^p_{x_3} (\Torus^3)}.
        \end{align*}
        We obtained (\ref{eq_estimates_nabla_nonlinear_terms2}).
    \end{proof}

    \begin{corollary} \label{cor_estimates_nabla_nonlinear_terms}
        Let $1 \leq p \leq \infty$, $T>0$.
        Then there exists a positive constant $C$ such that
        \begin{align*}
            & \sup_{0 < t < T} \left \Vert
                \int_0^{t/2}
                    \nabla e^{(t-s) \Delta} \mathbb{P}_\varepsilon \mathrm{div}_\varepsilon \left(
                        f (s) \otimes g (s)
                    \right)
                ds
            \right \Vert_{L^\infty_H L^p_{x_3} (\Torus^3)} \\
            &\leq C t^{-1/2} \min \left[
                \left(
                    \sup_{0 < t < T}  t^{1/2} \Vert
                        f (t)
                    \Vert_{L^\infty_H L^p_{x_3} (\Torus^3)}
                    + \sup_{0 < t < T}  t^{1/2} \Vert
                        \nabla f (t)
                    \Vert_{L^\infty_H L^p_{x_3} (\Torus^3)}
                \right)
            \right. \\
            & \quad \quad \quad \quad \quad \quad \quad
                \times \sup_{0 < t < T} \Vert
                    g (t)
                \Vert_{L^\infty_H L^p_{x_3} (\Torus^3)}, \\
            & \quad \quad \quad \quad \quad \quad \quad
                \sup_{0 < t < T}  \Vert
                    f (t)
                \Vert_{L^\infty_H L^p_{x_3} (\Torus^3)} \\
            & \quad \quad \quad \quad \quad \quad \quad \left.
                \times \left(
                    \sup_{0 < t < T}  t^{1/2} \Vert
                        g (t)
                    \Vert_{L^\infty_H L^p_{x_3} (\Torus^3)}
                    + \sup_{0 < t < T}  t^{1/2} \Vert
                        \nabla g (t)
                    \Vert_{L^\infty_H L^p_{x_3} (\Torus^3)}
                \right)
            \right] \\
            & + C t^{-1/2} \sup_{0 < t < T}  \Vert
                f (t)
            \Vert_{L^\infty_H L^p_{x_3} (\Torus^3)}
            \sup_{0 < t < T} t^{1/2}\Vert
                \nabla g (t)
            \Vert_{L^\infty_H L^p_{x_3} (\Torus^3)},
        \end{align*}
        and
        \begin{align*}
            & \sup_{0 < t < T} \left \Vert
                \int_{t/2}^t
                    \nabla e^{(t-s) \Delta} \mathbb{P}_\varepsilon \mathrm{div}_\varepsilon \left(
                        f (s) \otimes g (s)
                    \right)
                ds
            \right \Vert_{L^\infty_H L^p_{x_3} (\Torus^3)} \\
            & \leq C t^{- 1/2} \sup_{0 < t < T} \Vert
                \nabla f (t)
            \Vert_{L^\infty_H L^p_{x_3} (\Torus^3)} \\
            & \times \left(
                \sup_{0 < t < T} t^{1/2} \Vert
                    g (t)
                \Vert_{L^\infty_H L^p_{x_3} (\Torus^3)}
                + \sup_{0 < t < T} t^{1/2} \Vert
                    \nabla g (t)
                \Vert_{L^\infty_H L^p_{x_3} (\Torus^3)}
            \right)
        \end{align*}
        for all $\mathrm{div}_\varepsilon$-free vector fields $f, g \in L^\infty_t L^\infty_H L^p_{x_3} (\Torus^3 \times (0, T))$ satisfying $t^{\frac{1}{2}} \nabla f, t^{\frac{1}{2}} \nabla g \in L^\infty_t L^\infty_H L^p_{x_3} (\Torus^3\times (0, T))$, and $0 < \varepsilon <1$.
    \end{corollary}
    \begin{proof}
        The first inequality follows from (\ref{eq_estimates_nabla_nonlinear_terms1}) by integration over $(0,t)$.
        We apply the same way to (\ref{eq_estimates_nabla_nonlinear_terms2}) and the interval $(\frac{t}{2}, t)$ to get the second inequality.
    \end{proof}

\subsection{Estimates for $\tilde{F}(v, w)$}

    In this section, we establish $\varepsilon$-independent bounds in $L^\infty_H L^p_{x_3}  (\Torus^3)$ for
    \begin{align*}
        & t^{\alpha / 2} \partial_x^\alpha e^{t \Delta} \mathbb{P}_\varepsilon \tilde{F} (v,w)
        \quad \text{for} \quad \alpha = 0, 1.
    \end{align*}
    Note that by the assumption of Theorem \ref{thm_main} we see that
    \begin{align*}
        t^{\alpha/2} \Vert
            \nabla_H \nabla^\alpha v
        \Vert_{L^\infty_H L^p_{x_3}  (\Torus^3)}
        < \infty
        \quad \text{for} \quad \alpha = 0, 1.
    \end{align*}
    for the solution $v$ to (\ref{eq_primitive}).

    \begin{proposition} \label{prop_estimate_for_tilde_F_10}
        Let $\alpha=0,1$ and $T>0$.
        Then there exists a constant $C>0$ such that
        \begin{align} \label{eq_estimte_for_tilde_F_10}
            &\sup_{0<t<T} t^{\alpha/2} \left \Vert
                \int_0^t
                    \nabla^{\alpha} e^{(t -s) \Delta} \mathbb{P}_\varepsilon \int_{- \pi}^{x_3}
                        \left(
                            \begin{array}{c}
                                0\\
                                0\\
                                \mathrm{div}_H \left(
                                    f (s) \cdot \nabla g (s)
                                \right)
                            \end{array}
                        \right)
                    d z
                ds
            \right \Vert_{L^\infty_H L^p_{x_3}  (\Torus^3)}
            & \leq C, \\ \label{eq_estimte_for_tilde_F_15}
            &\sup_{0<t<T} t^{\alpha/2} \left \Vert
                \int_0^t
                    \nabla^{\alpha} e^{(t -s) \Delta} \mathbb{P}_\varepsilon \int_{- \pi}^{x_3}
                        \left(
                            \begin{array}{c}
                                0\\
                                0\\
                                \mathrm{div}_H \left(
                                    (\mathrm{div}_H f (s)) g (s)
                                \right)
                            \end{array}
                        \right)
                    d z
                ds
            \right \Vert_{L^\infty_H L^p_{x_3}  (\Torus^3)}
            & \leq C.
        \end{align}
        for all $0< \varepsilon \leq 1$ and two-dimensional vector fields $f, \, g \in C_t L^\infty_H L^p_{x_3} (\Torus^3 \times (0, T))$ satisfying
        \begin{align} \label{eq_regularity_assumption_for_vj}
            \sup_{0<t<T} t^{\alpha/2} \Vert
                \nabla_H \nabla^\alpha v_j (t)
            \Vert_{L^\infty_H L^p_{x_3} (\Torus^3)}
            < \infty,
            \quad j = 0,1.
        \end{align}
    \end{proposition}
    \begin{proof}
        We first consider the first inequality of the case $\alpha=0$.
        By Propositions \ref{prop_estimate_composit_operator}
        \begin{align*}
            & \left \Vert
                \int_0^t
                    e^{(t -s) \Delta} \mathbb{P}_\varepsilon \int_{- \pi}^{x_3}
                        \left(
                            \begin{array}{c}
                                0\\
                                0\\
                                \mathrm{div}_H \left(
                                    f (s) \cdot \nabla g (s)
                                \right)
                            \end{array}
                        \right)
                    d z
                ds
            \right \Vert_{L^\infty_H L^p_{x_3}  (\Torus^3)}\\
            & \leq C \int_0^t
                (t -s )^{- \frac{1}{2}} \left \Vert
                    f (s) \cdot \nabla_H g (s)
                \right \Vert_{L^\infty_H L^p_{x_3}  (\Torus^3)}
            ds \\
            & \leq C \int_0^t
                (t -s )^{- \frac{1}{2}} \left(
                    \Vert
                        f (s)
                    \Vert_{L^\infty_H L^p_{x_3}  (\Torus^3)}
                    + \Vert
                        \partial_3 f (s)
                    \Vert_{L^\infty_H L^p_{x_3}  (\Torus^3)}
                \right)
                \Vert
                    \nabla_H g (s)
                \Vert_{L^\infty_H L^p_{x_3}  (\Torus^3)}
            ds \\
            & \leq C \left(
                \sup_{0<t<T} t^{\frac{1}{2}}
                    \Vert
                        f (t)
                    \Vert_{L^\infty_H L^p_{x_3}  (\Torus^3)}
                + \sup_{0<t<T} t^{\frac{1}{2}}
                    \Vert
                        \partial_3 f (t)
                    \Vert_{L^\infty_H L^p_{x_3}  (\Torus^3)}
            \right)
            \sup_{0<t<T} \Vert
               \nabla_H g (t)
            \Vert_{L^\infty_H L^p_{x_3}  (\Torus^3)}.
        \end{align*}
        For $\alpha=1$ we see from Propositions \ref{prop_estimate_composit_operator} and the interpolation inequality for the horizontal derivative from Lemman 3.2 of Giga {\it et al.} \cite{GigaGriesHieberHusseinKashiwabara2017_L_infty_L1}
        \begin{align} \label{eq_interpolation_ineq_from_Giga_etal1}
            \Vert
                \nabla_H (- \Delta)^{-s/2} f
            \Vert_{L^\infty_H L^p_{x_3} (\Torus^3)}
            \leq C \Vert
                f
            \Vert_{L^\infty_H L^p_{x_3} (\Torus^3)}^{\frac{s}{2}}
            \Vert
                \nabla_H f
            \Vert_{L^\infty_H L^p_{x_3} (\Torus^3)}^{\frac{1 - s}{2}},
            \quad s \in (0,1),
        \end{align}
        that
        \begin{align*}
            & \left \Vert
                \nabla \int_0^t
                    e^{(t - s) \Delta} \mathbb{P}_\varepsilon \int_{- \pi}^{x_3}
                        \left(
                            \begin{array}{c}
                                0\\
                                0\\
                                \mathrm{div}_H \left(
                                    f (s) \cdot \nabla g (s)
                                \right)
                            \end{array}
                        \right)
                    d \zeta
                ds
            \right \Vert_{L^\infty_H L^p_{x_3}  (\Torus^3)}\\
            & \leq C \int_0^t
                (t - s)^{-\frac{1}{2} - \frac{1}{3}} \Vert
                    f (s) \cdot \nabla_H g (s)
                \Vert_{L^\infty_H L^p_{x_3}  (\Torus^3)}^{2/3}
                \Vert
                    \nabla_H \left(
                        f (s) \cdot \nabla_H g (s)
                    \right)
                \Vert_{L^\infty_H L^p_{x_3}  (\Torus^3)}^{1/3}
            ds \\
            & \leq C \int_0^t
                (t - s)^{-\frac{5}{6}} \left(
                    \Vert
                        f (s)
                    \Vert_{L^\infty_H L^p_{x_3}  (\Torus^3)}
                    + \Vert
                        \partial_3 f (s)
                    \Vert_{L^\infty_H L^p_{x_3}  (\Torus^3)}
                \right)^{2/3}
                \Vert
                    \nabla_H g (s)
                \Vert_{L^\infty_H L^p_{x_3}  (\Torus^3)}^{2/3} \\
            & \quad
                \times \left(
                    \Vert
                        \nabla_H f (s)
                    \Vert_{L^\infty_H L^p_{x_3}  (\Torus^3)}
                    + \Vert
                        \nabla_H \partial_3 f (s)
                    \Vert_{L^\infty_H L^p_{x_3}  (\Torus^3)}
                \right)^{1/3}
                \Vert
                        \nabla_H g (s)
                \Vert_{L^\infty_H L^p_{x_3}  (\Torus^3)}^{1/3}
            ds \\
            & + C \int_0^t
                (t - s)^{-\frac{5}{6}} \left(
                    \Vert
                        f (s)
                    \Vert_{L^\infty_H L^p_{x_3}  (\Torus^3)}
                    + \Vert
                        \partial_3 f (s)
                    \Vert_{L^\infty_H L^p_{x_3}  (\Torus^3)}
                \right)^{2/3}
                \Vert
                    \nabla_H g (s)
                \Vert_{L^\infty_H L^p_{x_3}  (\Torus^3)}^{2/3} \\
            & \quad
                \times \left(
                    \Vert
                        f (s)
                    \Vert_{L^\infty_H L^p_{x_3}  (\Torus^3)}
                    + \Vert
                        \partial_3 f (s)
                    \Vert_{L^\infty_H L^p_{x_3}  (\Torus^3)}
                \right)^{1/3}
                \Vert
                        \nabla_H^2 g (s)
                \Vert_{L^\infty_H L^p_{x_3}  (\Torus^3)}^{1/3}
            ds \\
            & \leq C t^{-1/2} \left(
                \sup_{0<t<T} t^{1/2} \Vert
                    f (t)
                \Vert_{L^\infty_H L^p_{x_3}  (\Torus^3)}
                + \sup_{0<t<T} t^{1/2} \Vert
                    \partial_3 f (s)
                \Vert_{L^\infty_H L^p_{x_3}  (\Torus^3)}
            \right)^{2/3} \\
            & \times \left(
                \sup_{0<t<T} \Vert
                    \nabla_H g (t)
                \Vert_{L^\infty_H L^p_{x_3}  (\Torus^3)}
            \right)^{2/3} \\
            & \times \left(
                \sup_{0<t<T} t^{1/2} \Vert
                    \nabla_H f (t)
                \Vert_{L^\infty_H L^p_{x_3}  (\Torus^3)}
                + \sup_{0<t<T} t^{1/2} \Vert
                    \nabla_H \partial_3 f (t)
                \Vert_{L^\infty_H L^p_{x_3}  (\Torus^3)}
            \right)^{1/3} \\
            & \times \left(
                \sup_{0<t<T} t^{1/2} \Vert
                    \nabla_H g (t)
                \Vert_{L^\infty_H L^p_{x_3}  (\Torus^3)}
            \right)^{1/3} \\
            & + C t^{-1/2} \left(
                \sup_{0<t<T} t^{1/2} \Vert
                    f (t)
                \Vert_{L^\infty_H L^p_{x_3}  (\Torus^3)}
                + \sup_{0<t<T} t^{1/2} \Vert
                    \partial_3 f (t)
                \Vert_{L^\infty_H L^p_{x_3}  (\Torus^3)}
            \right)^{2/3} \\
            & \times \left(
                \sup_{0<t<T} \Vert
                    \nabla_H g (t)
                \Vert_{L^\infty_H L^p_{x_3}  (\Torus^3)}
            \right)^{2/3} \\
            & \times \left(
                \sup_{0 < t < T} t^{1/2} \Vert
                    f (t)
                \Vert_{L^\infty_H L^p_{x_3}  (\Torus^3)}
                + \sup_{0<t<T} t^{1/2} \Vert
                    \partial_3 f (t)
                \Vert_{L^\infty_H L^p_{x_3}  (\Torus^3)}
            \right)^{1/3} \\
            & \times \left(
                \sup_{0<t<T} t^{1/2} \Vert
                    \nabla_H^2 g (t)
                \Vert_{L^\infty_H L^p_{x_3}  (\Torus^3)}
            \right)^{1/3}.
        \end{align*}
        This estimate implies (\ref{eq_estimte_for_tilde_F_10}).
        In the estimate (\ref{eq_estimte_for_tilde_F_15}), we change the role of $f$ and $g$ and use the same way as above, then we obtain (\ref{eq_estimte_for_tilde_F_15}).
    \end{proof}

    \begin{remark}
        In the proof of Proposition \ref{prop_estimate_for_tilde_F_10}, we used $\sup_{0<t<T} t^{1/2} \Vert \nabla_H^2 g (t) \Vert_{L^\infty_H L^p_{x_3}  (\Torus^3)}$.
        This is the why we imposed the additional regularity condition for $v_0$ in Theorem \ref{thm_main} with respect to the horizontal variable.
    \end{remark}

    \begin{proposition} \label{prop_estimates_for_tilde_F}
        Let $\alpha=0,1$ and $T>0$.
        Then there exists a constant $C>0$ such that
        \begin{align} \label{eq_estimte_for_tilde_F_20}
            &\sup_{0<t<T} t^{\alpha/2} \left \Vert
                \int_0^t
                    \nabla^{\alpha} e^{(t -s) \Delta} \mathbb{P}_\varepsilon \int_{- \pi}^{x_3}
                        \left(
                            \begin{array}{c}
                                0\\
                                0\\
                                \mathrm{div}_H \left(
                                    \int_{- \pi}^z
                                        \mathrm{div}_H f (s)
                                    d \zeta g (s)
                                \right)
                            \end{array}
                        \right)
                    dz
                ds
            \right \Vert_{L^\infty_H L^p_{x_3}  (\Torus^3)}
            \leq C,
        \end{align}
        for all two-dimensional vector fields $f, \, g \in L^\infty_H L^p_{x_3} (\Torus^3)$ satisfying (\ref{eq_regularity_assumption_for_vj}) and $0< \varepsilon \leq 1$.
    \end{proposition}
    \begin{proof}
        For $\alpha=0$ we see from Propositions \ref{prop_estimate_composit_operator}   that
        \begin{align*}
            &\sup_{0<t<T} \left \Vert
                    \int_0^t
                        e^{(t -s) \Delta} \mathbb{P}_\varepsilon \int_{- \pi}^{x_3}
                            \left(
                                \begin{array}{c}
                                    0\\
                                    0\\
                                    \mathrm{div}_H \left(
                                        \int_{- \pi}^z
                                            \mathrm{div}_H f (s)
                                        d \zeta g (s)
                                    \right)
                                \end{array}
                            \right)
                        dz
                    ds
                \right \Vert_{L^\infty_H L^p_{x_3}  (\Torus^3)} \\
                & \leq C \int_0^t
                    (t-s)^{-1/2} \Vert
                        \nabla_H f (s)
                    \Vert_{L^\infty_H L^p_{x_3}  (\Torus^3)}
                    \Vert
                        g (s)
                    \Vert_{L^\infty_H L^p_{x_3}  (\Torus^3)}
                ds \\
                & \leq C \sup_{0<t<T} t^{1/2} \Vert
                    \nabla_H f (t)
                \Vert_{L^\infty_H L^p_{x_3}  (\Torus^3)}
                \sup_{0<t<T} \Vert
                    g (t)
                \Vert_{L^\infty_H L^p_{x_3}  (\Torus^3)}.
        \end{align*}
        For $\alpha=1$ we see from Propositions \ref{prop_estimate_composit_operator} and (\ref{eq_interpolation_ineq_from_Giga_etal1}) that
        \begin{align*}
            &\sup_{0<t<T} \left \Vert
                \int_0^t
                    \nabla e^{(t -s) \Delta} \mathbb{P}_\varepsilon \int_{- \pi}^{x_3}
                        \left(
                            \begin{array}{c}
                                0\\
                                0\\
                                \mathrm{div}_H \left(
                                    \int_{- \pi}^z
                                        \mathrm{div}_H f (s)
                                    d \zeta g (s)
                                \right)
                            \end{array}
                        \right)
                    dz
                ds
            \right \Vert_{L^\infty_H L^p_{x_3} (\Torus^3)} \\
            & \leq C \int_0^t
                (t-s)^{-1/2- 1/3} \left \Vert
                    \int_{- \pi}^z
                        \mathrm{div}_H f (s)
                    d \zeta g (s)
                \right \Vert_{L^\infty_H L^p_{x_3} (\Torus^3)}^{2/3} \\
            & \quad \times \left \Vert
                    \nabla_H \left(
                        \int_{- \pi}^z
                            \mathrm{div}_H f (s)
                        d \zeta g (s)
                    \right)
                \right \Vert_{L^\infty_H L^p_{x_3} (\Torus^3)}^{1/3}
            ds \\
            & \leq C \int_0^t
            (t-s)^{-1/2- 1/3} \Vert
                    \nabla_H f (s)
                \Vert_{L^\infty_H L^p_{x_3}  (\Torus^3)}^{2/3}
                \Vert
                    g (s)
                \Vert_{L^\infty_H L^p_{x_3} (\Torus^3)}^{2/3} \\
            & \quad \times
                \left(
                    \Vert
                        \nabla_H^2 f (s)
                    \Vert_{L^\infty_H L^p_{x_3} (\Torus^3)}
                    \Vert
                        g (s)
                    \Vert_{L^\infty_H L^p_{x_3} (\Torus^3)}
                    + \Vert
                        \nabla_H f (s)
                    \Vert_{L^\infty_H L^p_{x_3} (\Torus^3)}
                    \Vert
                        \nabla_H g (s)
                    \Vert_{L^\infty_H L^p_{x_3} (\Torus^3)}
                \right)^{1/3}
            ds \\
            & \leq C t^{-1/2} \left(
                \sup_{0<t<T} t^{1/2} \Vert
                    \nabla_H f (t)
                \Vert_{L^\infty_H L^p_{x_3} (\Torus^3)}
            \right)^{2/3}
            \left(
                \sup_{0<t<T} t^{1/2} \Vert
                    g (t)
                \Vert_{L^\infty_H L^p_{x_3} (\Torus^3)}
            \right)^{2/3} \\
            & \quad \times \left(
                \sup_{0<t<T} t^{1/2} \Vert
                    \nabla_H^2 f (t)
                \Vert_{L^\infty_H L^p_{x_3} (\Torus^3)}
                \sup_{0<t<T} \Vert
                    g (t)
                \Vert_{L^\infty_H L^p_{x_3} (\Torus^3)}
            \right. \\
            & \quad \quad \quad \quad
            \left.
                + \sup_{0<t<T} t^{1/2} \Vert
                    \nabla_H f (t)
                \Vert_{L^\infty_H L^p_{x_3} (\Torus^3)}
                \sup_{0<t<T} \Vert
                    \nabla_H g (t)
                \Vert_{L^\infty_H L^p_{x_3} (\Torus^3)}
            \right)^{1/3}.
        \end{align*}
        Thus we have the conclusion.
    \end{proof}

    In the next section, we will use Propositions \ref{prop_estimate_for_tilde_F_10} and \ref{prop_estimates_for_tilde_F} to bound $\tilde{F}(v, w)$.

    \section{Estimates for the Solution to the Equation of Difference}
    By construction of the solution $u$ to the primitive equations in $C_t L^\infty_H L^p_{x_3} (\Torus^3 \times (0,T))$, see the proof of Theorem 2.1 of \cite{GigaGriesHieberHusseinKashiwabara2017_L_infty_L1}, the solution $u$ can be decomposed into
    \begin{align} \label{eq_decomposition_u}
        \begin{split}
            & u
            = u_{\text{smooth}}
            + u_{\text{small}} \\
            & u_{\text{smooth}}
            = (v_{\text{smooth}}, w_{\text{smooth}}), \quad
            u_{\text{small}}
            = (v_{\text{small}}, w_{\text{small}})
        \end{split}
    \end{align}
    at least in a short interval $(0,T_0)$, where
    \begin{align} \label{eq_estimates_decomposition_u}
        \begin{split}
            & \sup_{0<t<T_0} \Vert
                u_{\text{smooth}} (t)
            \Vert_{C^1 (\Torus^3)}
            \leq C, \\
            & \sup_{0<t<T_0} t^{\alpha/2} \Vert
                \nabla^\alpha u_{\text{small}} (t)
            \Vert_{L^\infty_H L^p_{x_3} (\Torus^3)}
            \leq \delta
        \end{split}
    \end{align}
    for $\alpha=0,1$, some constant $C>0$, and small $T_0, \delta>0$.
    Since $u$ is smooth for $t \geq T_0$, we can assume
    \begin{align} \label{eq_smoothness_u_after_T_0}
        \sup_{t \geq T_0} \Vert
            u (t)
        \Vert_{C^1 (\Torus^3)}
        \leq C
    \end{align}
    for some constant $C>0$.
    \begin{proof}[Proof of Theorem \ref{thm_main}]
        We put
        \begin{align*}
            & \widetilde{U}_\varepsilon
            = \left(
                \begin{array}{c}
                    V_\varepsilon \\
                    \varepsilon W_\varepsilon
                \end{array}
            \right), \quad
            \tilde{u}_\varepsilon
            = \left(
                \begin{array}{c}
                    v_\varepsilon \\
                    \varepsilon w_\varepsilon
                \end{array}
            \right) \\
            & \widetilde{u}_{\text{smooth}}
            = \left(
                \begin{array}{c}
                    v_{\text{smooth}} \\
                    \varepsilon w_{\text{smooth}}
                \end{array}
            \right), \quad
            \tilde{u}_{\text{small}}
            = \left(
                \begin{array}{c}
                    v_{\text{small}} \\
                    \varepsilon w_{\text{small}}
                \end{array}
            \right)
        \end{align*}
        and set
        \begin{align*}
            X_{\varepsilon, T} (\widetilde{U}_\varepsilon)
            & = \sup_{0<s<T} \Vert
                \widetilde{U}_\varepsilon (s)
            \Vert_{L^\infty_H L^p_{x_3} (\Torus^3)}
            + \sup_{0<s<T} s^{1/2} \Vert
                \nabla \widetilde{U}_\varepsilon (s)
            \Vert_{L^\infty_H L^p_{x_3} (\Torus^3)}, \\
            Y_T
            & = \sup_{0<s<T} \Vert
                u (s)
            \Vert_{L^\infty_H L^p_{x_3} (\Torus^3)}
            + \sup_{0<s<T} \Vert
                \nabla_H u (s)
            \Vert_{L^\infty_H L^p_{x_3} (\Torus^3)} \\
            & \quad + \sup_{0<s<T} s^{1/2} \Vert
                \nabla \nabla_H u (s)
            \Vert_{L^\infty_H L^p_{x_3} (\Torus^3)}.
        \end{align*}
        Note that these vector fields are $\mathrm{div}_\varepsilon$-free.
        We use the integral equations of the form
        \begin{align} \label{eq_integral_form_diff_eq_2}
            \begin{split}
                & \widetilde{U}_\varepsilon (t) \\
                & = \int_0^t
                    e^{(t-s) \Delta} \mathbb{P}_\varepsilon \mathrm{div}_\varepsilon    \left(
                        \widetilde{U}_\varepsilon (s) \otimes \widetilde{U} _\varepsilon (s)
                        + \widetilde{U}_\varepsilon (s) \otimes \tilde{u} (s)
                    \right)
                ds \\
                & = \int_0^t
                e^{(t-s) \Delta} \mathbb{P}_\varepsilon \left(
                    \widetilde{U}_\varepsilon (s) \cdot \nabla \tilde{u} (s)
                \right)
                ds \\
                & + \varepsilon \int_0^t
                e^{(t-s) \Delta} \mathbb{P}_\varepsilon \widetilde{F} (s)
                ds.
            \end{split}
        \end{align}
        This integral equations are equivalent to (\ref{eq_integral_form_diff_eq}).
        In view of (\ref{eq_decomposition_u}), the right-hand side is decomposed into
        \begin{align*}
            & \int_0^t
                e^{(t-s) \Delta} \mathbb{P}_\varepsilon \mathrm{div}_\varepsilon
                (
                    \widetilde{U}_\varepsilon (s) \otimes \widetilde{U}_\varepsilon (s)
                )
            ds\\
            & + \int_0^t
                e^{(t-s) \Delta} \mathbb{P}_\varepsilon \mathrm{div}_\varepsilon
                \left(
                    \widetilde{U}_\varepsilon (s) \otimes \tilde{u}_\text{smooth} (s)
                    + \widetilde{U}_\varepsilon (s)\otimes \tilde{u}_\text{small} (s)
                    + \tilde{u}_\text{small} (s) \otimes \widetilde{U}_\varepsilon (s)
                \right)
            ds \\
            & + \int_0^t
                e^{(t-s) \Delta} \mathbb{P}_\varepsilon \mathrm{div}_\varepsilon \left(
                    \tilde{u}_\text{smooth} (s) \otimes \widetilde{U}_\varepsilon (s)
                \right)
                ds \\
            & + \varepsilon \int_0^t
                e^{(t-s) \Delta} \mathbb{P}_\varepsilon \widetilde{F} (s)
            ds \\
            & =: N(\widetilde{U}_\varepsilon (t)).
        \end{align*}
        Propositions \ref{prop_estimate_for_tilde_F_10} and \ref{prop_estimates_for_tilde_F} imply
        \begin{align*}
            \left \Vert
                \varepsilon \int_0^t
                    \nabla^\alpha e^{(t-s) \Delta} \mathbb{P}_\varepsilon \widetilde{F} (s)
                ds
            \right \Vert_{L^\infty_H L^p_{x_3} (\Torus^3)}
            \leq C \varepsilon t^{\alpha/2} Y(t)^2
            \quad \text{for} \quad \alpha = 0,1.
        \end{align*}
        We see from Corollary \ref{cor_nonlinear_estimate_div_ep_free} and \ref{cor_estimates_nabla_nonlinear_terms} that the first and second terms of $N (\widetilde{U}_\varepsilon (t))$ is bounded by
        \begin{align*}
            & \left \Vert
                \int_0^t
                    \nabla^\alpha e^{(t-s) \Delta} \mathbb{P}_\varepsilon \mathrm{div}_\varepsilon \left(
                        \widetilde{U}_\varepsilon (s) \otimes \widetilde{U} _\varepsilon (s)
                        + \widetilde{U}_\varepsilon (s) \otimes \tilde{u}_\text{smooth} (s)
                    \right.
            \right. \\
            & \left. \quad \quad \quad \quad
                    \left.
                        + \widetilde{U}_\varepsilon (s) \otimes \tilde{u}_\text{small} (s)
                        + \tilde{u}_\text{small} (s) \otimes \widetilde{U}_\varepsilon (s)
                    \right)
                ds
            \right \Vert_{L^\infty_H L^p_{x_3} (\Torus^3)} \\
            & \leq C t^{\alpha/2} X_{\varepsilon, t} (\widetilde{U}_\varepsilon)^2
            + C t^{1/2 + \alpha/2} X_{\varepsilon, t} (\widetilde{U}_\varepsilon)
            + C t^{\alpha/2} \delta X_{\varepsilon, t} (\widetilde{U}_\varepsilon).
        \end{align*}
        Since
        \begin{align} \label{eq_docomposition_nonlinear_term_tmp}
            & \mathrm{div}_\varepsilon \left(
                \tilde{u}_\text{smooth} (s) \otimes \widetilde{U}_\varepsilon(s)
            \right) \notag \\
            & = \mathrm{div}_H \left(
                \tilde{u}_{smooth} \otimes \widetilde{V}_\varepsilon
            \right)
            - \partial_3 \left(
                \tilde{u}_{smooth} \int_{- \pi}^{x_3}
                    \mathrm{div}_H \widetilde{V}_\varepsilon
                dz
            \right),
        \end{align}
        we use interpolation inequalities (\ref{eq_interpolation_ineq_from_Giga_etal1}) and
        \begin{align*}
            \Vert
                \partial_3 \partial_3^{-s} f
            \Vert_{L^\infty_H L^p_{x_3} (\Torus^3)}
            \leq C \Vert
                f
            \Vert_{L^\infty_H L^p_{x_3} (\Torus^3)}^{\frac{s}{2}}
            \Vert
                \partial_3 f
            \Vert_{L^\infty_H L^p_{x_3} (\Torus^3)}^{\frac{1 - s}{2}},
            \quad s \in (0,1),
        \end{align*}
        which is a direct consequence of the one-dimensional interpolation inequality in $L^p (\Torus)$, and Proposition \ref{prop_estimate_composit_operator} to estimate the third term of $N(U_\varepsilon)$ as
        \begin{align*}
            & \left \Vert
                \int_0^t
                    e^{(t-s) \Delta} \mathbb{P}_\varepsilon \mathrm{div}_\varepsilon    \left(
                        \tilde{u}_\text{smooth} (s) \otimes \widetilde{U}_\varepsilon   (s)
                    \right)
                ds
            \right \Vert_{L^\infty_H L^p_{x_3} (\Torus^3)} \\
            & \leq C \int_0^t
                (t-s)^{-1/4} \left \Vert
                    \tilde{u}_\text{smooth} (s) \otimes \widetilde{V}_\varepsilon (s)
                \right \Vert^{1/2}
                \left \Vert
                    \nabla_H \left(
                        \tilde{u}_\text{smooth} (s) \otimes \widetilde{V}_\varepsilon   (s)
                    \right)
                \right \Vert_{L^\infty_H L^p_{x_3} (\Torus^3)}^{1/2}
            ds \\
            & + C \int_0^t
                (t-s)^{-1/4} \left \Vert
                    \tilde{u}_\text{smooth} (s) \int_{- \pi}^{x_3}
                        \mathrm{div}_H \widetilde{V}_\varepsilon (s)
                    dz
                \right \Vert_{L^\infty_H L^p_{x_3} (\Torus^3)}^{1/2} \\
            & \quad \quad \times
                \left \Vert
                    \partial_3 \left(
                        \tilde{u}_\text{smooth} (s) \int_{- \pi}^{x_3}
                            \mathrm{div}_H \widetilde{V}_\varepsilon (s)
                        dz
                    \right)
                \right \Vert_{L^\infty_H L^p_{x_3} (\Torus^3)}^{1/2}
            ds \\
            & \leq C t^{1/4} X_{\varepsilon, t} (\widetilde{U}_\varepsilon) \int_0^t
                (t-s)^{-1/4} s^{-1/2}
            ds \\
            & \leq C t^{1/4} X_{\varepsilon, t} (\widetilde{U}_\varepsilon),
        \end{align*}
        and similarly
        \begin{align*}
            & \left \Vert
                \nabla \int_0^t
                    e^{(t-s) \Delta} \mathbb{P}_\varepsilon \mathrm{div}_\varepsilon    \left(
                        \tilde{u}_\text{smooth} (s) \otimes \widetilde{U}_\varepsilon   (s)
                    \right)
                ds
            \right \Vert_{L^\infty_H L^p_{x_3} (\Torus^3)} \\
            & \leq C t^{1/4} X_{\varepsilon, t} (\widetilde{U}_\varepsilon) \int_0^t
                (t-s)^{-3/4} s^{-1/2}
            ds \\
            & \leq C t^{-1/4} X_{\varepsilon, t} (\widetilde{U}_\varepsilon),
        \end{align*}
        for some constant $C_3>0$.
        Summing up these estimates, we see that there exist a small $0<T_0<1$ and constants $C_0, C_1, C_2>0$ such that
        \begin{align}
            X_{\varepsilon, T_0} \left(
                N(\widetilde{U}_\varepsilon)
            \right)
            \leq C_2 X_{\varepsilon, T_0} (\widetilde{U}_\varepsilon)^2
            + C_1 (T_0^{1/4}+\delta) X_{\varepsilon, T_0} (\widetilde{U}_\varepsilon)
            + C_0 Y_{T_0}^2 \varepsilon.
        \end{align}
        Thus if we take $\varepsilon$ and $T_0$ so small that
        \begin{align*}
            C_1 (T_0^{1/4} + \delta)
            < 1, \quad
            \varepsilon
            < \frac{\left(
                    1 - (T_0^{1/4} + \delta)
                \right)^2
            }{
                4 C_0 C_2 Y_{T_0}^2
            },
        \end{align*}
        we obtain
        \begin{align*}
            X_{\varepsilon, T_0} \left(
                N(\widetilde{U}_\varepsilon)
            \right)
            \leq 2 \varepsilon C_0 Y_{T_0}^2,
        \end{align*}
        for $X_{\varepsilon, T_0} (\widetilde{U}_\varepsilon)
        \leq 2 \varepsilon C_0 Y_{T_0}^2$.
        We consider the difference
        \begin{align*}
            \widetilde{U}_\varepsilon^\prime (s)
            & := \widetilde{U}_{\varepsilon}^2 - \widetilde{U}_{\varepsilon}^1
        \end{align*}
        for $\widetilde{U}_{\varepsilon}^1, \widetilde{U}_{\varepsilon}^2$ satisfying $X_{\varepsilon, t} (\widetilde{U}_\varepsilon^j) < \infty$ for $j=1,2$.
        By the same way as above, we have
        \begin{align*}
            & X_{\varepsilon, T_0} \left(
                N(\widetilde{U}_\varepsilon^\prime)
            \right) \\
            & \leq C_2^\prime X_{\varepsilon, T_0} (\widetilde{U}_\varepsilon) \left(
                X_{\varepsilon, T_0} (\widetilde{U}_\varepsilon^1)
                + X_{\varepsilon, T_0} (\widetilde{U}_\varepsilon^2)
            \right)
            + C_1^\prime (T_0^{1/4}+\delta) X_{\varepsilon, T_0} (\widetilde{U}_\varepsilon).
        \end{align*}
        Thus, for sufficiently small $\varepsilon, T_0$, and $\delta$, we see $N$ is a contraction map.
        By the contraction mapping principle, we see that there exists a unique solution $\widetilde{U}_\varepsilon \in C_t C_H L^p_{x_3} (\Torus^3 \times [0, T_0])$ to (\ref{eq_integral_form_diff_eq_2}) such that
        \begin{align*}
            X_{\varepsilon, T_0} (\widetilde{U}_\varepsilon)
            \leq 2 \varepsilon C_0 Y_{T_0}^2.
        \end{align*}
        We next consider the integral equations with initial data $\widetilde{U}_\varepsilon (T_0)$ such as
        \begin{align} \label{eq_int_eq_after_T_0}
            \begin{split}
                &\widetilde{U}_\varepsilon (t) \\
                &= e^{t \Delta} \widetilde{U}_\varepsilon (T_0) \\
                &+ \int_0^t
                    e^{(t-s) \Delta} \mathbb{P}_\varepsilon \mathrm{div}_\varepsilon    \left(
                        \widetilde{U}_\varepsilon (s) \otimes \widetilde{U} _\varepsilon (s)
                        + \widetilde{U}_\varepsilon (s) \otimes \tilde{u} (s+T_0)
                    \right)
                ds \\
                &= \int_0^t
                e^{(t-s) \Delta} \mathbb{P}_\varepsilon \left(
                    \widetilde{U}_\varepsilon (s) \cdot \nabla \tilde{u} (s+T_0)
                \right)
                ds \\
                &+ \int_0^t
                e^{(t-s) \Delta} \mathbb{P}_\varepsilon \widetilde{F} (s+T_0)
                ds.
            \end{split}
        \end{align}
        Because of the bound (\ref{eq_smoothness_u_after_T_0}), things are much easier.
        We can use the same way as above estimates to see that there exist small $0<T_1<1$ and constant $C_3, C_4, C_5>0$ such that
        \begin{align} \label{eq_contraction}
            \begin{split}
                & X_{\varepsilon, T_1} \left(
                    N(\widetilde{U}_\varepsilon)
                \right)\\
                & \leq C_5 X_{\varepsilon, T_1} (\widetilde{U}_\varepsilon)^2
                + C_4 T_1^{1/4} X_{\varepsilon, T_1} (\widetilde{U}_\varepsilon)
                + C_3 \left(
                    \varepsilon Y_{T_0+T_1}^2
                    + \Vert
                        \widetilde{U}_\varepsilon (T_0)
                    \Vert_{L^\infty_H L^p_{x_3} (\Torus^3)}
                \right) \\
                & \leq C_5 X_{\varepsilon, T_1} (\widetilde{U}_\varepsilon)^2
                + C_4 T_1^{1/4} X_{\varepsilon, T_1} (\widetilde{U}_\varepsilon)
                + C_3 \varepsilon \left(
                    Y_{T_0+T_1}^2
                    + 2 C_0 Y_{T_0}^2
                \right).
            \end{split}
        \end{align}
        Thus if $\varepsilon$ so small that
        \begin{align*}
            C_4 T_1^{1/4}
            < 1, \quad
            \varepsilon
            < \frac{
                \left(
                    1 - T_1^{1/4}
                \right)^2
            }{
                2 C_3 \left(
                    Y_{T_0+T_1}^2
                    + 2 C_0 Y_{T_0}^2
                \right) C_5
            },
        \end{align*}
        we obtain
        \begin{align*}
            X_{\varepsilon, T_1} \left(
                N(\widetilde{U}_\varepsilon)
            \right)
            \leq 2 \varepsilon C_3 \left(
                Y_{T_0+T_1}^2
                + 2 C_0 Y_{T_0}^2
            \right)
        \end{align*}
        for $X_{\varepsilon, T_1} (\widetilde{U}_\varepsilon)
        \leq 2 \varepsilon C_0 Y_{T_1}$.
        By the same way as (\ref{eq_contraction}), we see that $N$ is a contraction mapping for small $T_1$ and $\varepsilon$.
        Note that the constant $T_1$ is independent of $\varepsilon$.
        We use the contraction mapping principle again to get a unique solution $\widetilde{U}_\varepsilon \in C^\infty_t C_H L^p_{x_3} (\Torus^3 \times [0, T_1])$ to (\ref{eq_int_eq_after_T_0}) such that
        \begin{align*}
            X_{\varepsilon, T_0} (\widetilde{U}_\varepsilon)
            \leq 2 \varepsilon C_3 \left(
                Y_{T_0+T_1}^2
                + 2 C_0 Y_{T_0}^2
            \right).
        \end{align*}
        Since $T_1$ is independent of $\varepsilon$, if we choose $\varepsilon$ sufficiently small beforehand, we can repeat the above procedures up to $T$.
        We proved Theorem \ref{thm_main}.
    \end{proof}

\begin{appendices}
\section{Derivation of the equation for $w$} \label{seq_derivation_PE}
    Here we derive the equation (\ref{eq_for_w}).
    Let
    \begin{align*}
        \overline{v} (x^\prime)
        := \frac{1}{2} \int_{- \pi}^\pi
            f (x^\prime, z)
        dz
        \quad \text{and} \quad
        \widetilde{v}
        := f - \overline{f}
    \end{align*}
    for any $x^\prime \in \Torus^2$ and integrable function $f$.
    It is clear that
    \begin{align} \label{eq_average_of_tilde_is_zero}
        \int_{- \pi}^1 \tilde{f}(\cdot, z) dz
        = 0.
    \end{align}
    We see that
    \begin{align*}
        \mathrm{div}_H g
        = \mathrm{div}_H \, \widetilde{g},
    \end{align*}
    and
    \begin{align*}
        w(\cdot, x_3)
        = \int_{- \pi}^{x_3}
            \mathrm{div}_H \, g
        dz
        = \int_{- \pi}^{x_3}
            \mathrm{div}_H \, \widetilde{g}
        dz,
    \end{align*}
    for any $- \pi \leq x_3 \leq \pi$ and integrable $\mathrm{div}_H$-free vector $g$.
    The first equation of (\ref{eq_primitive}) is equivalent to
    \begin{align} \label{eq_primitive_2}
        \partial_t v - \Delta v
        + \tilde{v} \cdot \nabla_H \tilde{v}
        + \overline{v} \cdot \nabla_H \tilde{v}
        + \tilde{v} \cdot \nabla_H \overline{v}
        + \overline{v} \cdot \nabla_H \overline{v}
        + w \partial_3 \tilde{v}
        + \nabla_H \pi
        = 0.
    \end{align}
    Applying $\frac{1}{2 \pi} \int_{- \pi}^\pi \, \cdot \, dz$ to the both sides of (\ref{eq_primitive_2}), we have
    \begin{align} \label{eq_averaged_eq_for_PE}
        \partial_t \overline{v} - \Delta \overline{v} + \frac{1}{2 \pi} \int_{- \pi}^\pi
            \tilde{v} \cdot \nabla_H \tilde{v}
        dz
        + \overline{v} \cdot \nabla_H \overline{v}
        + \frac{1}{2 \pi} \int_{-\pi}^\pi
            w \partial_z \tilde{v}
        dz
        + \nabla_H \pi
        = 0.
    \end{align}
    Note that the boundary traces from $- \Delta v$ vanish since $v$ is a even vector field with respect to $x_3$.
    Taking the difference between (\ref{eq_primitive_2}) and (\ref{eq_averaged_eq_for_PE}), we have a nonlinear parabolic equation
    \begin{align}
        & \partial_t \tilde{v} - \Delta \tilde{v}
        + \tilde{v} \cdot \nabla_H \tilde{v}
        + \overline{v} \cdot \nabla_H \tilde{v}
        + \tilde{v} \cdot \nabla_H \overline{v}
        -  \frac{1}{2\pi} \int_{- \pi}^\pi
            \tilde{v} \cdot \nabla_H \tilde{v}
        dz \notag \\
        & \quad \quad \quad \quad \quad + w \partial_3 \tilde{v}
        - \frac{1}{2\pi} \int_{- \pi}^\pi
            w \partial_3 \tilde{v}
        dz
        = 0.
    \end{align}
    Integration by parts and the formula $\partial_z w = - \mathrm{div}_H \tilde{v}$ lead to
    \begin{align}
        \begin{split}
        & \partial_t \tilde{v} - \Delta \tilde{v}
            + \tilde{v} \cdot \nabla_H \tilde{v}
            + \overline{v} \cdot \nabla_H \tilde{v}
            + \tilde{v} \cdot \nabla_H \overline{v}
            -  \frac{1}{2\pi} \int_{- \pi}^\pi
                \tilde{v} \cdot \nabla_H \tilde{v}
            dz \\
            & \quad \quad \quad \quad \quad + w \partial_3 \tilde{v}
            - \frac{1}{2 \pi} \int_{- \pi}^\pi
                (\mathrm{div}_H \tilde{v}) \tilde{v}
            dz
            = 0.
        \end{split}
    \end{align}
    Applying $\int_{- \pi}^{x_3} \mathrm{div_H} \, \cdot \, d \zeta$ to the both sides, we have
    \begin{align} \label{eq_for_w_pre}
        \begin{split}
            & \partial_t w - \Delta w + \int_{- \pi}^{x_3}
                \mathrm{div}_H \left(
                    \tilde{v} \cdot \nabla_H \tilde{v}
                    + \overline{v} \cdot \nabla_H \tilde{v}
                    + \tilde{v} \cdot \nabla_H \overline{v}
                \right)
            dz \\
            & - \frac{1}{2 \pi} (x_3 - 1) \int_{- \pi}^\pi
                \mathrm{div}_H \left(
                    \tilde{v} \cdot \nabla_H \tilde{v}
                \right)
            dz \\
            & + \int_{- \pi}^{x_3}
                \mathrm{div}_H \left(
                    w \partial_z \tilde{v}
                \right)
            dz
            - \frac{1}{2 \pi} (x_3 - \pi) \int_{- \pi}^\pi
                \mathrm{div}_H \left[
                    (\mathrm{div}_H \tilde{v}) \tilde{v}
                \right]
            dz
            = 0.
        \end{split}
    \end{align}
    Integration by parts yields
    \begin{align*}
        \int_{- \pi}^{x_3}
            w \partial_z \tilde{v}
        dz
        & = w \tilde{v}
        + \int_{- \pi}^{x_3}
            (\mathrm{div}_H \tilde{v}) \tilde{v}
        dz \\
        & = - \int_{- \pi}^{x_3}
            \mathrm{div}_H \tilde{v}
        dz \, \tilde{v}
        + \int_{- \pi}^{x_3}
            (\mathrm{div}_H \tilde{v}) \tilde{v}
        dz.
    \end{align*}
    Thus (\ref{eq_for_w_pre}) is equivalent to
    \begin{align*}
        & \partial_t w - \Delta w + \int_{- \pi}^{x_3}
            \mathrm{div}_H \left(
                \tilde{v} \cdot \nabla_H \tilde{v}
                + \overline{v} \cdot \nabla_H \tilde{v}
                + \tilde{v} \cdot \nabla_H \overline{v}
            \right)
        dz
        - \frac{1}{2} (x_3 - \pi) \int_{- \pi}^1
            \mathrm{div}_H \left(
                \tilde{v} \cdot \nabla_H \tilde{v}
            \right)
        dz\\
        & + \int_{- \pi}^{x_3}
                \mathrm{div}_H \left(
                - \int_{- \pi}^{z}
                    \mathrm{div}_H \, \tilde{v}
                d \zeta \, \tilde{v}
            \right)
        dz
        + \int_{- \pi}^{x_3}
            \mathrm{div}_H \left[
                (\mathrm{div}_H \, \tilde{v}) \tilde{v}
            \right]
        dz \\
        & - \frac{1}{2} (x_3 - \pi) \int_{- \pi}^1
            \mathrm{div}_H \left[
                (\mathrm{div}_H \tilde{v}) \tilde{v}
            \right]
        dz
        = 0,
    \end{align*}
    which is (\ref{eq_for_w}).

    \section{Decomposition of the solution around initial time}
    In this appendix we briefly show that the solution to (\ref{eq_primitive}) can be decomposed such that (\ref{eq_decomposition_u}) and (\ref{eq_estimates_decomposition_u}).
    The proof is quite similar to the proof Theorem \ref{thm_main}, we do not repeat the things for simplicity.
    We decomposed the initial data $v_0$ satisfying the assumptions of Theorem \ref{thm_main} such that
    \begin{align*}
        v_0
        = v_{0, smooth}
        + v_{0, small}
    \end{align*}
    satisfying
    \begin{align} \label{eq_decomposition_initial_data_v_0}
        \begin{split}
            \Vert
                v_{0, smooth}
            \Vert_{C^2(\Torus^3)}
            \leq C, \quad
            \Vert
                \nabla_H^\alpha v_{0, small}
            \Vert_{L^\infty_H L^1_{x_3} (\Torus^3) }
            \leq \delta^\ast,
        \end{split}
    \end{align}
    for some constant $C>0$, $\alpha=0,1$ and small $\delta> 0$.

    We know that there exist $T > 0$ and a unique solution $u^\ast = (v^\ast, w^\ast) \in C_t C^1_x (\Torus^3 \times (0, T)) \times C_t C_x (\Torus^3 \times (0, T))$ to (\ref{eq_primitive}) satisfying
    \begin{align} \label{eq_bound_for_u_ast}
        \Vert
            v^\ast
        \Vert_{C_t C^1(\Torus^3 \times [0, T])}
        \leq C^\ast
    \end{align}
    for some constant $C^\ast>0$ and small $T>0$.
    The reader refers to \cite{GigaGriesHieberHusseinKashiwabara2017_analiticity}.
    Let $P$ be the hydrostatic Helmholtz projection on $\Torus^3$.
    \begin{proposition} \label{prop_estimate_hydrostatic_composite_operator}
        Let $t>0$, $0 \leq \beta < 1$, and $p \geq 1$.
        Then there exits a constant $C>0$ such that
        \begin{align} \label{eq_estimates_hydrostatic_composite_operator}
            \begin{split}
                & \Vert
                    e^{t \Delta} P \mathrm{div} \left(
                        f \otimes g
                    \right)
                \Vert_{L^\infty_H L^p_{x_3}(\Torus^3)} \\
                & \leq C t^{- \frac{1 - \beta}{2}} \left(
                    \Vert
                        \nabla_H g
                    \Vert_{L^\infty(\Torus^3)}
                    \Vert
                        f
                    \Vert_{L^\infty_H L^p_{x_3}(\Torus^3)}
                    + \Vert
                        g_H
                    \Vert_{L^\infty(\Torus^3)}
                    \Vert
                        \nabla f
                    \Vert_{L^\infty_H L^p_{x_3}(\Torus^3)}
                \right)^\beta \\
                & \times \left(
                    \Vert
                        g_H
                    \Vert_{L^\infty(\Torus^3)}
                    \Vert
                        f
                    \Vert_{L^\infty_H L^p_{x_3}(\Torus^3)}
                \right)^{1 - \beta} \\
                & + C \min \left(
                        t^{- \frac{1 - \beta}{2}}
                        \Vert
                            \nabla_H g_H
                        \Vert_{L^\infty(\Torus^3)}^\beta \left(
                            \Vert
                                f
                            \Vert_{L^\infty_H L^p_{x_3}(\Torus^3)}
                            + \Vert
                                \partial_3 f
                            \Vert_{L^\infty_H L^p_{x_3}(\Torus^3)}
                        \right)^\beta
                \right. \\
                & \left.
                        \quad \quad \quad\quad \quad \times \left(
                        \Vert
                            \nabla_H g_H
                        \Vert_{L^\infty(\Torus^3)}
                        \Vert
                            f
                        \Vert_{L^\infty_H L^p_{x_3}(\Torus^3)}
                    \right)^{1 - \beta},
                    \Vert
                        \nabla_H g_H
                    \Vert_{L^\infty_H L^p_{x_3}(\Torus^3)}
                    \Vert
                        f
                    \Vert_{L^\infty(\Torus^3)}
                \right)
            \end{split}
        \end{align}
        for all two-dimensional vector fields $f \in C^1(\Torus^3)$ and divergence-free $g \in L^\infty_H L^p_{x_3}(\Torus^3)$ satisfying $\nabla u \in L^\infty_H L^p_{x_3}(\Torus^3)$.
    \end{proposition}
    \begin{proof}
        The proof is essentially same as Lemma 6.1 of \cite{GigaGriesHieberHusseinKashiwabara2017_L_infty_L1}.
        We know from the Lemma that
        \begin{align} \label{eq_P_div_tilde_v_u}
            P \mathrm{div} (f \otimes u)
            = P \mathrm{div}_H (g_H \otimes f)
            + \partial_3 (g_3 f).
        \end{align}
        We apply the interpolation inequality to find
        \begin{align*}
            & \Vert
                e^{t \Delta} P \mathrm{div}_H \left(
                    g_H \otimes f
                \right)
            \Vert_{L^\infty_H L^p_{x_3}(\Torus^3)} \\
            & \leq C t^{- \frac{1 - \beta}{2}}
                \Vert
                    \nabla_H (g_H \otimes f)
                \Vert_{L^\infty(\Torus^3)}^\beta
                \Vert
                    g_H \otimes f
                \Vert_{L^\infty_H L^p_{x_3}(\Torus^3)}^{1 - \beta} \\
            & \leq C t^{- \frac{1 - \beta}{2}} \left(
                \Vert
                    \nabla_H g_H
                \Vert_{L^\infty(\Torus^3)}
                \Vert
                    f
                \Vert_{L^\infty_H L^p_{x_3}(\Torus^3)}
                + \Vert
                    g_H
                \Vert_{L^\infty(\Torus^3)}
                + \Vert
                    \nabla_H f
                \Vert_{L^\infty_H L^p_{x_3}(\Torus^3)}
            \right)^\beta \\
            & \times \left(
                \Vert
                    g_H
                \Vert_{L^\infty(\Torus^3)}
                \Vert
                    f
                \Vert_{L^\infty_H L^p_{x_3}(\Torus^3)}
            \right)^{1 - \beta}.
        \end{align*}
        The seconde term in (\ref{eq_P_div_tilde_v_u}) is bounded as
        \begin{align*}
            & \Vert
                e^{t \Delta}  \partial_3 \left(
                    g_3 f
                \right)
            \Vert_{L^\infty_H L^p_{x_3}(\Torus^3)} \\
            & \leq C t^{- \frac{1 - \beta}{2}}
                \Vert
                    \partial_3 (g_3 f)
                \Vert_{L^\infty(\Torus^3)}^\beta
                \Vert
                    g_3 f
                \Vert_{L^\infty_H L^p_{x_3}(\Torus^3)}^{1 - \beta} \\
            & \leq C t^{- \frac{1 - \beta}{2}} \Vert
                \nabla_H g_H
            \Vert_{L^\infty(\Torus^3)}^\beta \left(
                \Vert
                    f
                \Vert_{L^\infty_H L^p_{x_3}(\Torus^3)}
                + \Vert
                    \partial_3 f
                \Vert_{L^\infty_H L^p_{x_3}(\Torus^3)}
            \right)^\beta \\
            & \times \left(
                \Vert
                    \nabla_H g_H
                \Vert_{L^\infty(\Torus^3)}
                \Vert
                    f
                \Vert_{L^\infty_H L^p_{x_3}(\Torus^3)}
            \right)^{1 - \beta}
        \end{align*}
        and
        \begin{align*}
            & \Vert
                e^{t \Delta}  \partial_3 \left(
                    g_3 f
                \right)
            \Vert_{L^\infty_H L^p_{x_3}(\Torus^3)} \\
            & = \left \Vert
                e^{t \Delta} \left(
                    (\mathrm{div}_H g_H) f
                    + \int_{- \pi}^{x_3}
                        \mathrm{div}_H g_H
                    dz \partial_3 f
                \right)
            \right \Vert_{L^\infty_H L^p_{x_3}(\Torus^3)} \\
            & \leq C \Vert
                \nabla_H g_H
            \Vert_{L^\infty_H L^p_{x_3}(\Torus^3)}
            \Vert
                f
            \Vert_{L^\infty(\Torus^3)}.
        \end{align*}
        Thus we have (\ref{eq_estimates_hydrostatic_composite_operator}).
        Note that the Proposition \ref{prop_estimate_hydrostatic_composite_operator} also holds if we change the role of $g_H$ and $\tilde{f}$.
    \end{proof}

    We now show the decomposition (\ref{eq_decomposition_u}).
    We only consider the case $\alpha = 0$ in (\ref{eq_decomposition_u}) for simplicity.
    Since $v_0$ and also $v_{0, small}$ has more regularity for the horizontal direction, it not difficult to improve the regularity to the case $\alpha = 1$.
    Put
    \begin{align*}
        N (v^\ast, v)
        := e^{t \Delta} v_{0, small}
        - \int_0^t
            e^{(t -s) \Delta} P \mathrm{div} \left(
                u (s) \otimes v (s)
                + u^\ast (s) \otimes v (s)
                + u (s) \otimes v^\ast (s)
            \right)
        ds,
    \end{align*}
    where $u = (v, w)$ and $w$ is give by (\ref{eq_formula_for_w}).
    To show the decomposition, it is enough to show that there exists a solution to the equation $v = N (v^\ast, v)$ satisfying the second estimate of (\ref{eq_estimates_decomposition_u}).

    We apply Proposition 6.2 in \cite{GigaGriesHieberHusseinKashiwabara2017_L_infty_L1}, see also the proof of Theorem 2.1, to get
    \begin{align*}
        & \left \Vert
            \int_0^t
                e^{(t -s) \Delta} P \mathrm{div} \left(
                    u (s) \otimes v (s)
                \right)
            ds
        \right \Vert_{L^\infty_H L^p_{x_3} (\Torus^3)} \\
        & \leq C \sup_{0 < s < t} \Vert
            v (s)
        \Vert_{L^\infty_H L^p_{x_3} (\Torus^3)}
        \sup_{0 < s < t} s^{\frac{1}{2}} \Vert
            \nabla v (s)
        \Vert_{L^\infty_H L^p_{x_3} (\Torus^3)},
    \end{align*}
    and
    \begin{align*}
        & \left \Vert
            \int_0^t
                \nabla e^{(t -s) \Delta} P \mathrm{div} \left(
                    u (s) \otimes v (s)
                \right)
            ds
        \right \Vert_{L^\infty_H L^p_{x_3} (\Torus^3)} \\
        & \leq C t^{- \frac{1}{2}} \sup_{0 < s < t} s^{\frac{1}{2}} \left(
            \Vert
                v (s)
            \Vert_{L^\infty_H L^p_{x_3} (\Torus^3)}
            + \Vert
                \nabla v (s)
            \Vert_{L^\infty_H L^p_{x_3} (\Torus^3)}
        \right)^{\frac{3}{2}}
        \left(
            \sup_{0 < s < t} s^{\frac{1}{2}} \Vert
                \nabla v (s)
            \Vert_{L^\infty_H L^p_{x_3} (\Torus^3)}
        \right)^{\frac{1}{2}}.
    \end{align*}
    Proposition \ref{prop_estimate_hydrostatic_composite_operator} and (\ref{eq_bound_for_u_ast}) imply
    \begin{align*}
        & \left \Vert
            \int_0^t
                e^{(t -s) \Delta} P \mathrm{div} \left(
                    u^\ast (s) \otimes v (s)
                    + u (s) \otimes v^\ast (s)
                \right)
            ds
        \right \Vert_{L^\infty_H L^p_{x_3} (\Torus^3)} \\
        & \leq C t^{\frac{1}{2}} C^\ast \sup_{0 < s < t} \left(
            \Vert
                v (s)
            \Vert_{L^\infty_H L^p_{x_3} (\Torus^3)}
            + \sup_{0 < s < t} s^{\frac{1}{2}} \Vert
                \nabla v (s)
            \Vert_{L^\infty_H L^p_{x_3} (\Torus^3)}
        \right),
    \end{align*}
    and
    \begin{align*}
        & \left \Vert
            \int_0^t
                \nabla e^{(t -s) \Delta} P \mathrm{div} \left(
                    u^\ast (s) \otimes v (s)
                    + u (s) \otimes v^\ast (s)
            \right)
            ds
        \right \Vert_{L^\infty_H L^p_{x_3} (\Torus^3)} \\
        & \leq C C^\ast \left(
            t^{\frac{1}{4}} \sup_{0 < s < t} \Vert
                v (s)
            \Vert_{L^\infty_H L^p_{x_3} (\Torus^3)}
            + t^{- \frac{1}{4}} \sup_{0 < s < t} s^{\frac{1}{2}} \Vert
                \nabla v (s)
            \Vert_{L^\infty_H L^p_{x_3} (\Torus^3)}
        \right)^{\frac{1}{2}} \\
        & \quad \quad \quad \times \left(
            \sup_{0 < s < t} \Vert
                v (s)
            \Vert_{L^\infty_H L^p_{x_3} (\Torus^3)}
        \right)^{\frac{1}{2}} \\
        & + C C^\ast \sup_{0 < s < t} s^{\frac{1}{2}} \Vert
            \nabla v (s)
        \Vert_{L^\infty_H L^p_{x_3} (\Torus^3)},
    \end{align*}
    where we took $\beta = 0, 1/2$ for the first and second estimates, respectively.
    If we set
    \begin{align*}
        X_T (v)
        = \sup_{0 < s < T} \Vert
            v (s)
        \Vert_{L^\infty_H L^p_{x_3} (\Torus^3)}
        + \sup_{0 < s < T} s^{\frac{1}{2}} \Vert
            \nabla v (s)
        \Vert_{L^\infty_H L^p_{x_3} (\Torus^3)},
    \end{align*}
    the above estimates lead the quadratic estimate
    \begin{align*}
        X_T (N(v^\ast, v))
        \leq C_2 X_T(v)^2
        + C_1 T^{\frac{1}{4}} X_T(v)
        + C_0 \delta
    \end{align*}
    for some constants $C_0, C_1, C_2$ and small $0 < T < 1$.
    If we take $T$ and $\delta$ sufficiently small beforehand, $X_T (N(v^\ast, v))$ can be bounded small for small $v$.
    Since this argument is same as the proof of Theorem \ref{thm_main}, we omit details here.
    By the similar way, we see that $N$ becomes a contraction mapping for small $v$.
    Thus we can obtain the desired solution by the contraction mapping principle.
\end{appendices}
    
\end{document}